\documentclass[12pt]{amsart}

\usepackage[ansinew]{inputenc}
\usepackage{pst-all,epsfig}

\usepackage{float}

\usepackage{graphicx, graphics}

\usepackage[USenglish]{babel}

\usepackage{amsmath}
\usepackage{amsfonts}
\usepackage{amssymb}
 
\usepackage[left=2.5cm,top=2cm,right=2.5cm,bottom=2.5cm]{geometry}
\newcommand{\R}{\mathbb{R}}
\newcommand{\Rd}{\mathbb{R}^2}
\newcommand{\Rt}{\mathbb{R}^3}
\newcommand{\Rn}{\mathbb{R}^{n}}

\newcommand{\Sd}{\mathbb{S}^{2}}

\newcommand{\inte}{\operatorname*{int}}

\newcommand{\conv}{\operatorname*{conv}}
\newcommand{\bd}{\operatorname*{bd}}
\newcommand{\dims}{\operatorname*{dim}}

 \newcommand{\len}{\operatorname*{length}}

\newcommand{\lin}{\operatorname*{lin}}

\newtheorem{lemma}{Lemma}
\newtheorem{theorem}{Theorem}

\newtheorem{corollary}{Corollary}

\title{A characterization of translated convex bodies}

\author{E. Morales-Amaya \\
}
\thanks{This research was supported by the National Council of Humanities Sciences and Technology of Mexico (CONAHCyT), SNI 21120. This work was done partially during a sabbatical visit of the author at University College London (UCL). The author thanks UCL for their hospitality and support.}
\address{Faculty of Mathematics-Acapulco, Autonomous University of Guerrero, Carlos  E. Adame 54, Col. Garita C.P. 39650, Acapulco, Guerrero. Mexico.}
\date{March 10, 2025.}

\begin{document}

\maketitle

\begin{abstract} In this work we present a theorem concerning two convex bodies $K_1, K_2\subset \mathbb{R}^{n}$, where $n\geq 3$, and two families of their sections determined by tangent planes of two spheres 
$S_i\subset \inte K_i$, $i=1,2$. The theorem considers pairs of parallel supporting planes $\Pi_1$, $\Pi_2$ for $S_1$ and $S_2$, respectively, 
such that these planes correspond in the following sense: the outer normal vectors of the supporting half-spaces determined by $\Pi_1$, $\Pi_2$ are identical in direction. Under this condition, if the sections $\Pi_1\cap K_1$, $\Pi_2\cap K_2$ are translations of each other for every such pair of corresponding planes, then the theorem states: (1) if $S_1$, $S_2$ have the same radius, then the bodies $K_1$ and $K$ are translations of each other; (2) if $S_1$, $S_2$ have distinct radii, then both $K_1$ and $K$ must be spheres.
 \end{abstract}
\section{Introduction.}  
In \cite{ro1}, A. Rogers proved that if every pair of parallel 2-sections of two convex bo\-dies $K_1,K_2$ passing through two fix points $p_1\in \inte K_1, p_2\in \inte K_2$ are directly homothetics (i.e. such section are homothetic and the ratio of the similitud is positive), the convex bo\-dies are directly homo\-thetic. On the other hand, in \cite {burton}, G. R. Burton proved the general case $p_1, p_2\in  \mathbb{R}^{n}$. An interesting variation of Roger's Theorem was presented in \cite{mo}, there the two families of sections are not longer given by concurrent planes, instead L. Montejano considers two families of planes which, for one hand, vary continuously and, on the other hand, given a direction $v$, there is only one plane of each family orthogonal to $v$, however, Montejano restricts himself to consider only translated sections. The case for homothetic sections was considered in \cite{jmm2}. We can find several interesting problems and result about the determination of a convex bodies by families of section given by concurrent hyperplanes 
in \cite{bl}, \cite{gardner} and \cite{dima}. On the other hand, in \cite{kuru1} and \cite{kuru2} there are results that characterize the sphere in terms of information on the measurement of either the solid angles determined by the support cones of a convex body or the sections of a convex body tangent to some sphere.

In order to present formally our main result, which can be con\-si\-de\-red either as generalization of  A. Roger's theorem given in \cite{ro1} or as geometrical progress of a conjecture due to J. A. Barker and D. G. Larman \cite{bl}, we need the following notation and definitions.  

Let $\mathbb{R}^{n}$ be the Euclidean space of dimension $n$ endowed with the usual inner product $\langle \cdot, \cdot\rangle : \mathbb{R}^{n} \times \mathbb{R}^{n} \rightarrow \R$. We take a orthogonal system of coordinates $(x_1,...,x_{n})$ for  $\mathbb{R}^{n}$. Let $B_r(n)=\{x\in \mathbb{R}^{n}: ||x||\leq r\}$ be the $n$-ball of radius $r$ centered  at the origin, and let $r\mathbb{S}^{n-1}=\{x\in \mathbb{R}^{n}: ||x|| = r\}$ be its boundary. For each vector $u\in \mathbb{S}^{n}$, we denote by $H^ {+}(u)$  the closed half-space $\{x \in \mathbb{R} ^{n}: x\cdot u \leq 0\}$ with unit normal vector $u$, by $H(u)$ its boundary hyperplane $\{x \in \mathbb{R} ^{n}: \langle x,  u\rangle = 0\}$. For $r\in \mathbb{R}$, we denote by $G(u)$, $rG(u)$ the affine hyperplanes $u+H(u)$, $ru+H(u)$ and by $E(u)$, $rE(u)$, $F(u)$, $rF(u)$ the half-spaces $u+H^+(u)$, $ru+H^+(u)$, $u+H^+(-u)$, $ru+H^+(-u)$, respectively (See Fig. \ref{vamos}). 

Let $S_1,S_2\subset \Rn$ be two spheres with centers $p_1,p_2$ and radius $r_1,r_2$, respectively. The parallel supporting planes $\Pi_1,\Pi_2$ of $S_1, S_2$, respectively, are said to be \textit{corresponding} if there exist $u\in \mathbb S^{n-1}$ such that 
\[
\Pi_i=p_i+r_iG(u)\textrm{      }\textrm{      }\textrm{     and    }\textrm{      }\textrm{      }S_i\subset p_i+r_iE(u),\textrm{      }i=1,2. \textrm{      }(\textrm{See}\textrm{      } \textrm{Fig.} \textrm{      }\ref{negrita})
\]
Let $W_1,W_2\subset \Rn$ be two convex bodies, i.e, compact convex sets with non empty interior. The bodies $W_1,W_2$ are said to be \textit{translated} if there exists a non-zero vector $u$ such that $W_2=u+W_2$.
\begin{theorem} \label{paris} 
Let $K_i\subset \mathbb{R}^{n}$  be a convex body, $n\geq 3$, and let $S_i\subset \inte K_i$  be a sphere with center $p_i$ and radius $r_i$, $i=1,2$ .  Suppose that, for every pair of corresponding planes $\Pi_1, \Pi_2$ of $S_1$ and $S_2$, there exists a translation $\psi: \mathbb{R}^{n} \rightarrow \mathbb{R}^{n}$ such that
\[
\psi(\Pi_2\cap K_2) = \Pi_1\cap K_1.
\]	 
If $r_1=r_2$, then $K_1, K_2$ are translated. If $r_1\not=r_2$,  then $K_1$ and $K_2$ are spheres with centers at $p_1$ and $p_2$, respectively. 
\end{theorem}
\begin{center} 
\begin{figure}[h]
\hspace{1cm}
\includegraphics[scale=.6]{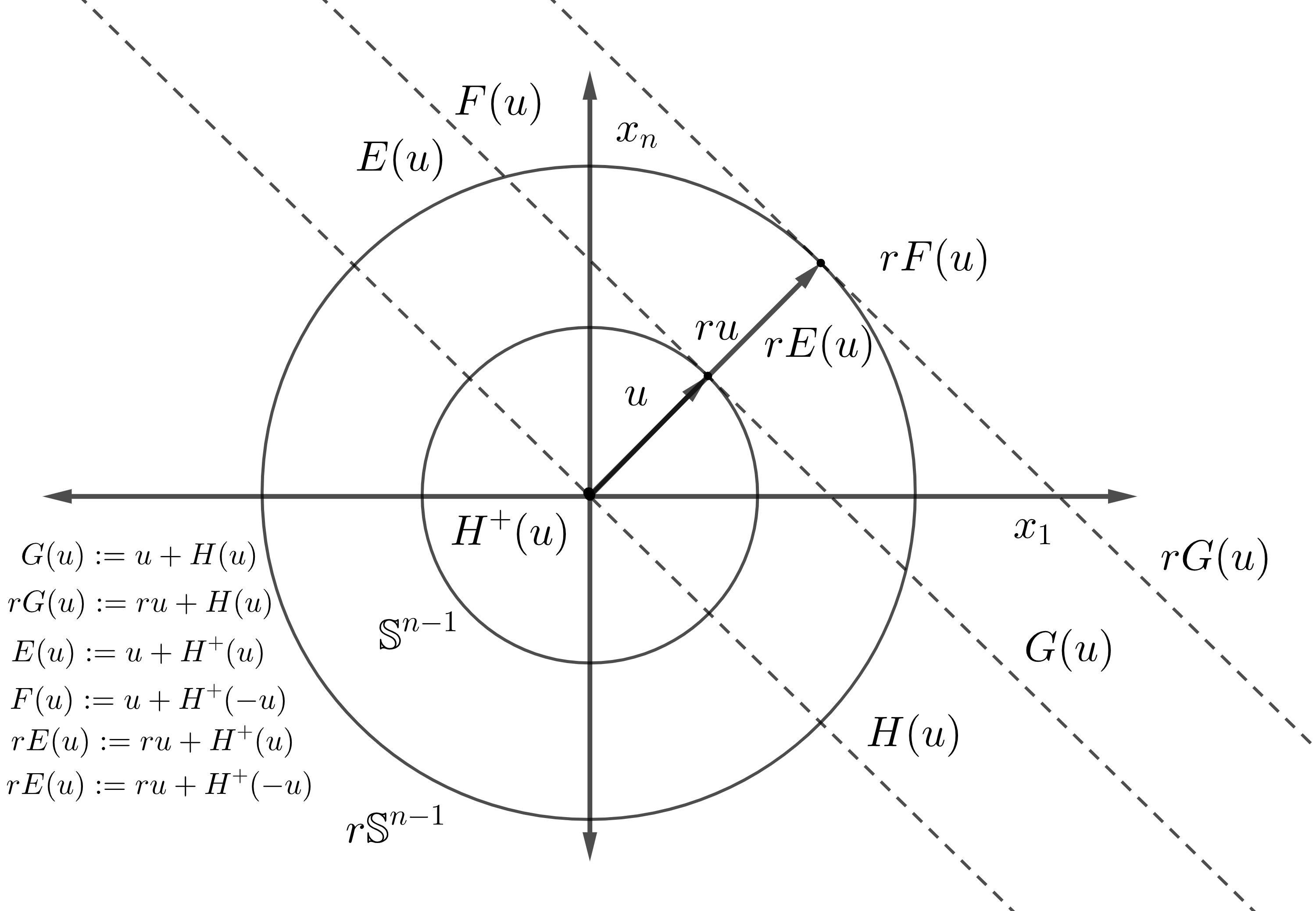}
\caption{Definition of the affine hyperplanes $G(u)$, $rG(u)$ and the half-spaces $E(u)$, $rE(u)$, $F(u)$, $rF(u)$.}
\label{vamos}
\end{figure}
\end{center}
\begin{center} 
\begin{figure}[h]
\hspace{1cm}
\includegraphics[scale=.5]{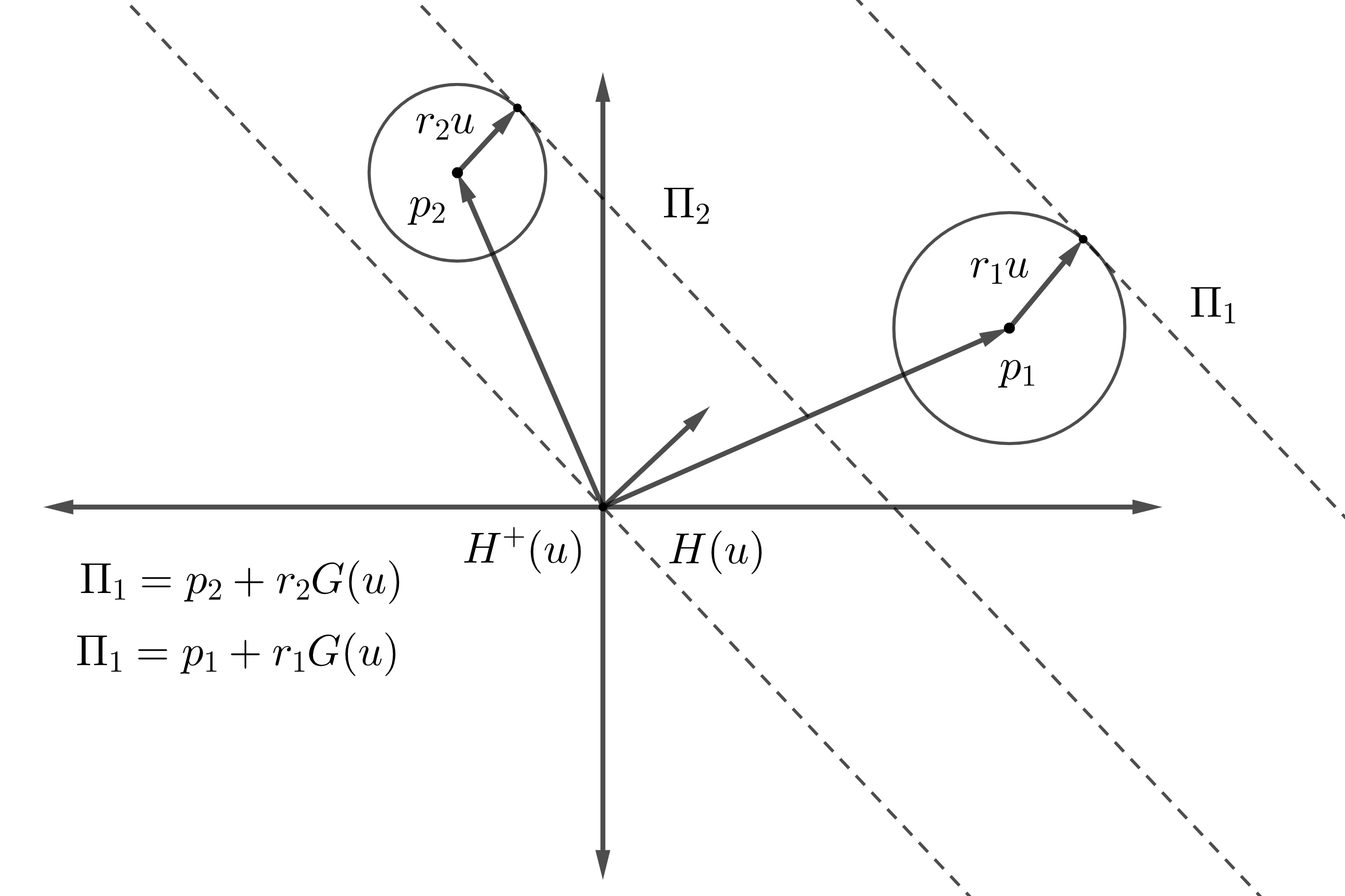}
\caption{A pair $\Pi_1, \Pi_2$ of Corresponding planes.}
\label{negrita}
\end{figure}
\end{center}
\begin{corollary}\label{francesita}
Let $K \subset \mathbb{R}^n$ be a convex body, $n\geq 3$, and let $S\subset K$ be a $(n-1)$-sphere with center $p$. Suppose that: 1) for all supporting hyperplane $\Pi$ of $S$ the section $\Pi \cap K$ is centrally symmetric and 2) for all pairs of parallel supporting hyperplanes $\Pi_1, \Pi_2$ of $S$ the sections $\Pi_1 \cap K$ and $ \Pi_2\cap K$ are translated. Then $K$ is centrally symmetric.
\end{corollary}
\begin{proof}  
We take a system of coordinates such that $p$ is the origin. We apply Theorem \ref{paris}, taking $K_1=K, K_2=-K$ and $S_1=S=S_2$. Therefore $K$ and $-K$ are translated. Thus $K$ is centrally symmetric.  
\end{proof}  
 This work is organized as following:
\begin{enumerate}
\item Introduction.

\item Reduction of the Theorem \ref{paris} to dimension 3.

\item Characterization of the sphere.

\item Lemmas for the case $n=3$ and $r_{1}= r_{2}$. 

\item Proof of the Theorem \ref{paris} in dimension 3 for $r_1=r_2$. 

\item Lemmas for the case $n=3$, $r_{1}\not =r_{2}$.

\item Proof of the Theorem \ref{paris} in dimension 3  for $r_1\not=r_2$. 

\end{enumerate}
\section{Reduction of Theorem \ref{paris} to dimension $n=3$.}  
For the points $x,y \in \Rn$ we will denote by $xy$ the line determined by $x$ and $y$ and by $[x,y]$ the line segment contained in $L(x,y)$ with endpoints $x$ and $y$.  As usual $\inte K$, $\bd K$ will denote the interior and the boundary of the convex body $K$, respectively.  For each vector $u\in \mathbb{S}^{n-1}$, we denote by $\pi_u:\Rn \rightarrow H(u)$ the orthogonal projection parallel to $u$.

Suppose that the bodies $K_1,K_2$ and the spheres $S_1,S_2$ satisfy the conditions of Theorem \ref{paris} for dimension $n, n\geq 4$, and that the Theorem \ref{paris}
holds in dimension $n-1 $. We are going to prove that, for $u\in \mathbb{S}^{n}$, the bodies $\pi_u(K_1)$, $\pi_u(K_1)$ and the spheres $\pi_u(S_1)$, $\pi_u(S_2)$ satisfy the condition of the Theo\-rem \ref{paris} in dimension $n-1$. Consequently, either, if $r_1=r_2$, $\pi_u(K_1)$, $\pi_u(K_1)$ are translated or, if $r_1\not=r_2$, $\pi_u(K_1)$, $\pi_u(K_1)$ are spheres. Thus if all the orthogonal projections of $K_1$ and $K_2$  are translated, in virtue of Theorem 1 of \cite{ro1}, $K_1$ and $K_2$ are translated, otherwise, $K_1$ and $K_2$ are spheres, i.e, the Theorem \ref{paris} holds for dimension $n$.

Notice that the conditions of Theorem \ref{paris} are invariant under translation, consequently we can assume that the centres $p_1,p_2$ are at origin of a system of coordinates. Let $v\in H(u)\cap \mathbb{S}^{n}$. Since $K_1\cap r_1G(v)$ is a translated copy of  $K_2\cap r_2G(v)$, i.e., there exists a vector $\omega \in \Rn$ such that
\[
K_1\cap r_1G(v)= \omega + [K_2\cap r_2G(v)],
\]
then 
\[
\pi_u(K_1\cap r_1G(v)) =\pi_u(\omega) + \pi_u(K_2\cap r_2G(v)),\]
 i.e., 
\[
\pi_u(K_1)\cap  \pi_u(r_1G(v)) =\pi_v(\omega) + \pi_u(K_2)\cap   \pi_u(r_2G(v)), \textrm{ i.e.},
\]
$\pi_u(K_1)\cap \pi_u(r_1G(v))$ is a translated copy of $\pi_u(K_2)\cap \pi_u(r_2G(v))$ for every $v\in H(u)\cap \mathbb{S}^{n}$, that is, the bodies $\pi_u(K_1)$, $\pi_u(K_1)$ and the spheres $\pi_u(B_1)$, $\pi_u(B_2)$ satisfies the condition of the Theo\-rem \ref{paris} in dimension $n-1$.
\section{Characterization of the sphere.}
Given a convex body $K\subset \mathbb{R}^n$ and a point $x\in\mathbb R^n\setminus K$ we denote the cone generated by $K$ with apex $x$ by $\text{C}(K,x)$, i.e., $\text{C}(K,x):=\{x+\lambda (y-x): y\in K,\, \lambda\geq 0\}.$ The boundary of $\text{C}(K,x)$ is denoted by $\text{S}(K,x)$, in other words, $\text{S}(K,x)$ is the support cone of $K$ from the point $x$. We denote the \emph{graze} of $K$ from $x$ by $\Sigma(K,x)$, i.e., $\Sigma(K,x):=\text{S}(K,x)\cap \bd K.$ On the other hand, given a line $L$ (or a vector $u$), we denote by $\text{C}(K,L)$  the cylinder generated by $K$ and $L$, i.e., the family of all the lines parallel to $L$ and with non empty intersection with $K$. The boundary of $\text{C}(K,L)$ is denoted by $\text{S}(K,L)$, in other words, $\text{S}(K,L)$ is the support cylinder of $K$ corresponding to $L$. We denote the \emph{shadow boundary} of $K$ corresponding to $L$ by $\Sigma(K,L)$, i.e., $\Sigma(K,L):=\text{S}(K,L)\cap \bd K.$  
\begin{lemma}\label{chaparrito} 
Let $M_1, M_2 \subset \mathbb{R}^{3} $ be convex bodies, let $B$ be a sphere with 
$B\subset M_1 \cap M_2 $ and let $r$ be a real number, $r>0, r\not=1$.  Suppose that for each supporting plane 
$\Pi$ of $B$ the sections $\Pi \cap M_1$ and $\Pi \cap M_2$ are homothetic, with centre and radius of homothety $\Pi \cap B$ and $r$, respectively. Then $M_1$ and $M_2$ are spheres  concentric with 
$B$.
\end{lemma}
\begin{proof}
We claim that $\bd M_1   \cap \bd M_2=\emptyset$. Suppose that there is a point $p\in \bd M_1 \cap \bd M_2$. Let $\Pi$ be a supporting plane of $B$, $p\in \Pi$ and it touches the boundary of $B$ in the point $x$. Since $B\subset M_1 \cap M_2$, we have $x\in \inte (\Pi \cap M_1), \inte (\Pi \cap M_2)$. If either $\inte (\Pi \cap M_1) \subset \inte (\Pi \cap M_2) $ or $\inte (\Pi \cap M_2) \subset \inte (\Pi \cap M_1)$, then either the centre of homothety $y$ which sends $\Pi \cap M_1$ into $\Pi \cap M_2$ is in a common supporting line of $\Pi \cap M_2$ and $ \Pi \cap M_2$, and therefore $y\not= x$, or the radius $r$ is equal to 1, but, in the first case, we contradicts the hypothesis, i.e., $x$ is the centre of homothety between the sections $\Pi \cap M_1$ and $\Pi \cap M_2$, on the other hand, if $r=1$ then $M_1=M_2$. The case $\inte (\Pi \cap M_1)$ not contained in $\inte (\Pi \cap M_2)$, or viceversa, implies that the centre of homothety is a point outside of the sections $\Pi \cap M_1$, $\Pi \cap M_2$ but again this is in contradiction with the hypothesis that $x\in \inte M_1 \cap \inte M_2$ is the centre of homothety between this two sections.
 \begin{figure}[h]
\hspace{1cm}
\includegraphics[scale=.3]{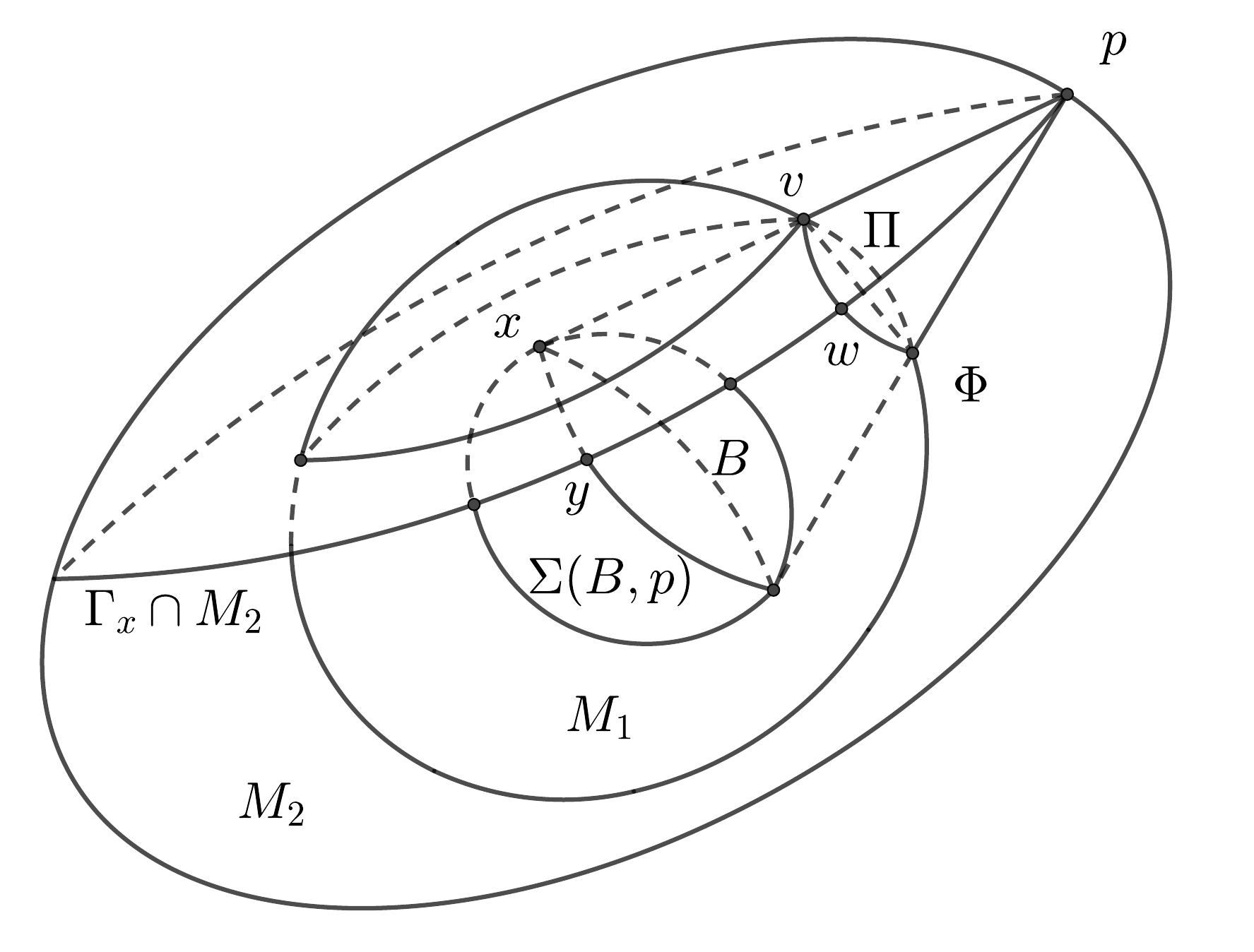}
\caption{Characterization of two concentric spheres by means of homothetic sections tangent to a sphere.}
\label{anidados}
\end{figure}
 Thus we can choose the notation such that $B\subset M_1\subset M_2$. Let $p\in \bd M_2$.  We take one the two components of $S(B,p) \cap \bd M_1$ and we denote it by $\Phi$. We are going to show that 
$\Phi$ is a planar curve contained in a plane parallel to the plane where $\Sigma(B,p)$ is contained (See Fig. \ref{anidados}). From this, it follows that $\Phi$ is a circle. Let $v\in \Phi$ and let $\Pi$ be a plane parallel to the plane of $\Sigma(B,p)$ and $v\in \Pi$. Pick an arbitrary point $w\in \Phi, w\not= v$. Let $x,y\in \Sigma(B,p)$ be points where the supporting lines $pv$ and $pw$ of $B$ touches $\bd B$, respectively. By virtue of the hypothesis, it is easy to see that $xp/xv=yp/yw=r$ (Taking supporting planes $\Gamma_x$ and $\Gamma_y$ of $B$ at $x$ and $y$, respectively, we have that the sections $ \Gamma_x \cap M_1$, $ \Gamma_x \cap M_2$ are homothetics with centre and radius of homothety $x$ and $r$, respectively, and the sections $\Gamma_y \cap M_1$, $ \Gamma_y \cap M_2$ are homothetics with centre and radius of homothety $y$ and $r$, respectively). Thus the triangles $\triangle xpy$ and $\triangle vpw$ are similar and consequently, the lines $xy$ and $vw$ are parallel. Hence $w\in \Pi$. It follows that $M_1$ is a sphere concentric with $B$; in \cite{bg} is proved a more general statement: \textit{if $K\subset \mathbb{R}^n$ is a convex body and there is a family of hyperplanes $\mathcal{F}$, which vary continuously, such that the sections defined by the hyperplanes of $\mathcal{F}$ are $(n-1)$-ellipsoid, then $K$ is an $n$-ellipsoid}.    

Since $M_1$ is a sphere all the sections of $M_2$ with planes tangent to $B$ are circles, then, $M_2$ is a sphere concentric with $B$.
\end{proof}
  
\section{Lemmas for the case $n=3$ and $r_1=r_2$.}  
In this section we will assume that $n=3$ and $r_1=r_2$. Notice that the conditions of Theorem \ref{paris} are invariant under translations, consequently, we can assume that $S_1=S_2$ and that we have a system of coordinates such that $p_1$ is the origin and $r_1=1$. Thus $S_1=\mathbb{S}^2$. For $u\in \mathbb{S}^{2}$, we denote by $K_i (u)$ the section $G(u)\cap K_i, i=1,2$. By virtue of the theorem's hypothesis there exists a map $\mu: \mathbb{S}^2 \rightarrow \textrm{T} \mathbb{S}^2$ such that $\mu(u) \cdot u=0$ and $K_2 (u)= \mu(u)+K_1(u)$, where $T \mathbb{S}^2$ is the tangent bundle of $\mathbb{S}^2$.
From the continuity of the boundaries of $K_1$ and $K_2$ it follows the continuity of $\mu(u)$. Hence, by virtue of Theorem 27.8 P. 141 of \cite{steenrod},  there exists $u_1\in \mathbb{S}^2$ such that
\begin{eqnarray}\label{melones}
K_2 (u_1)= K_1(u_1).
\end{eqnarray}
Let $\Omega \subset \mathbb{S}^2$ be the collection of vectors $u$ such that $u\in \Omega $ if 
\[
K_1(u)= K_2(u)
\]
and let $\Phi \subset \mathbb{S}^2$ the set of vectors such that 
$u\in \Phi$ if the line $G(u_1)\cap G(u)$ is supporting line of $K_1(u_1)$. The curve $\Phi$ is homeomorphic to $\mathbb{S}^1$, then, by the Jordan's Theorem,  $\mathbb{S}^2$ is decomposed by $\Phi$ in two disjoint open sets, say $A,B$ and we choose the notation such that $u_1\in A$. In order to prove the Theorem \ref{paris} first, we are going to show now that $\Phi \subset \Omega$ and, second, that  $\Omega = \mathbb{S}^2$. Therefore $K_1=K_2$. 
\begin{lemma}\label{caboe} 
The set $\Phi$ is contained in  $\Omega$. 
\end{lemma}
\begin{proof}
Let $u\in \Phi$, i.e., $L(u):=G(u_1)\cap G(u)$ is supporting line of $K_1(u_1)$. We are going to show that the sections $K_1(u)$ and $K_2(u)$ coincide. We claim that $K_1(u)$ and $K_2(u)$ are contained in the same half plane of $G(u)$ of the two determined by $L(u)$. Let $\Psi_i$ be a supporting planes of $K_i$  containing $L(u)$, $i=1,2$. We denote by $\Psi^{+}_i$, the  supporting half spaces determined by $\Psi_i$, $i=1,2$. It is clear that $\mathbb{S}^2 \subset \Psi^{+}_1 \cap \Psi^{+}_2$, since we are assuming that $\mathbb{S}^2 \subset K_1  \cap K_2$. Hence $K_1(u)$, $ K_2(u) \subset \Psi^{+}_1 \cap \Psi^{+}_2$ and, consequently, $K_1(u)$, $K_2(u)$ are contained in the same half plane of $G(u)$ of the two determined by $L(u)$. On the other hand, we denote by $S$ the intersection $L(u) \cap   K_1$. Observe that, since (\ref{melones}) holds, $S$ is equal to $L(u) \cap   K_2$. By hypothesis the sections $K_1(u)$ and $ K_2(u)$ are translated, by virtue that they have a common supporting line, the set $S$ in common and they are contained in the same half plane determined by $L(u)$, they coincide.
\end{proof}  
Let $u,v\in \Phi$ such that $L(u):=G(u)\cap G(u_1)$ and $L(v):=G(v)\cap G(u_1)$ are parallel supporting lines of $K_1(u_1)$. For $u\in \Phi$, we define the following subset of $\bd K_1$   
\[
S(u):=[E(u)\cap  F(v) \cap E(u_1)]\cap \Sigma(K_1,L(u)).
\]  
Now we consider the union of the sets $S(u)$ for $u\in \Phi$, i.e.,
\begin{eqnarray}\label{ana}
\Sigma:=\bigcup_{u\in \Phi} S(u).
\end{eqnarray}    
\begin{figure}[h]
\hspace{1cm}
\includegraphics[scale=1.3]{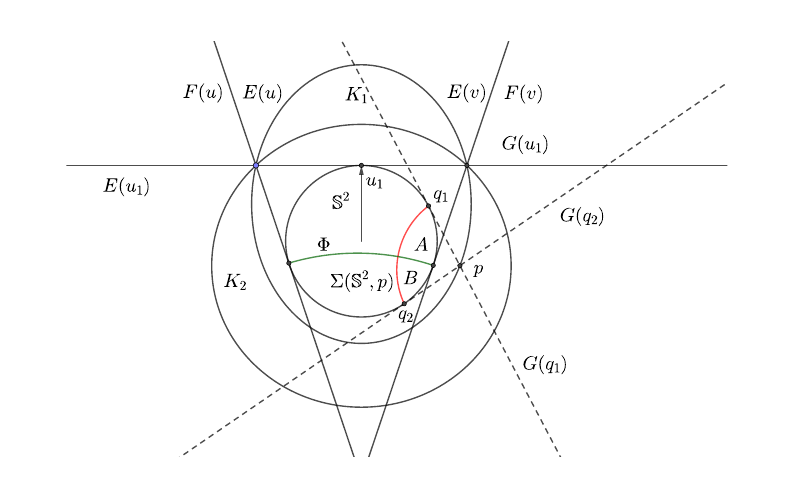}
\caption{The set $\Sigma$ is contained in $\bd K_2$.}
\label{sacre}
\end{figure}
Let $k$ be an integer, $1\leq k \leq n-1$. An \textit{embedding} of $\mathbb{S}^{k}$ in $\Rn$ is a map $\alpha: \mathbb{S}^{k} \rightarrow \Rn$ such that $\alpha$ is homeomorphism onto its image.
\begin{lemma}\label{sara}
The set $\Sigma$ is contained in $\bd K_2$.
\end{lemma}
\begin{proof} 
Let $p$ be a point in $\Sigma$. Then there are parallel supporting lines $L(u),L(v)$ of $K_1(u_1)$ such that 
$p\in  S(u)$ for some $u\in \Phi$, i.e.,  there exists $u\in \Phi$ such that $p\in [E(u)\cap  F(v) \cap E(u_1)]\cap \Sigma(K_1,L(u))$. First we assume that $p$ is not in $G(-u_1)$. We denote by $D$ the central projection of $\mathbb{S}^2$ onto the plane $G(u_1)$ from the point $p$, i.e., $D:=S(\mathbb{S}^2,p)\cap G(u_1)$. By virtue of the choice of $p$, it follows that $\Phi \cap \Sigma(\mathbb{S}^2, p)\not= \emptyset$. Since $p\in [E(u)\cap  F(v) \cap E(u_1)]\cap \Sigma(K_1,L(u))$ there exists a supporting line $L\subset E(u)\cap  F(v) \cap E(u_1)$ of $K_1$ through $p$ and parallel to $L(u)$. Let $G(q_1), G(q_2) $ be the two supporting planes  of $\mathbb{S}^2$ containing $L$. Changing the notation if necessary, we can assume that $q_1\in A$ and $q_2\in B$ (See Fig. \ref{sacre}). Therefore in each one of the arcs of $\Sigma(\mathbb{S}^2,p)$ determined by $q_1,q_2$ there is a point in $\Phi$, say $w_1,w_2$. Consequently, the figures $D$ and $K_1(u_1)$ have two common supporting lines, say $L(w_1):=G(w_1)\cap G(u_1), L(w_2):=G(w_2)\cap G(u_1)$. By Lemma \ref{caboe}, $K_i(w_i)= K_2 (w_i), i=1,2$. Hence $p\in \bd K_2$. On the other hand, if $p\in E(-u_1)$, from the previous case and since the boundaries of $K_1,K_2$ are closed it follows that $p\in K_2$. 
\end{proof}
\begin{lemma}\label{parcera}
Let $M_1,M_2\subset \Rd$ be two translated convex bodies and let $a,b,c,d$ be points which belongs to $\bd M_1\cap \bd M_2$. If the quadrilateral $\square abcd$ is not a trapezoid, then $\bd M_1=\bd M_2$. On the other hand, if the quadrilateral $\square abcd$ is a trapezoid and $\bd M_1\not= \bd M_2$, then the pair of parallel edges of $\square abcd$ is contained in $\bd M_1$ and $\bd M_2$.
\end{lemma}
\begin{proof}
Since the convex figures $M_1,M_2\subset \Rd$ are translated there exists $w\in \Rd$ such that $M_1=w+M_2$.

Let $abcd \subset \mathbb{R}^2$ be a quadrilateral which is not a trapezoid.
Let $u\in \mathbb{S}^1$ be a unit vector, not parallel to an edge of $abcd$. Then there exists two vertices of $abcd$, say $b,d$,  such that 
\begin{eqnarray}\label{cinthya}
[b+r^+(u)]\cap \inte (abcd)\not= \emptyset \textrm{ }\textrm{and } \textrm{ }[d+r^{-}(u)]\cap \inte (abcd)\not= \emptyset
\end{eqnarray}
where $r^+(u):=\{\alpha u: \alpha \in \mathbb{R}, \alpha>0 \}$ and $r^{-}(u):=\{\alpha u: \alpha \in \mathbb{R}, \alpha<0 \}$ are the rays defined by $u$. 

Since $\inte (abcd) \subset \inte M_1 \cap \inte K_2$, the relation (\ref{cinthya}) implies that $w$ is not parallel to $u$. O\-ther\-wise, let us assume that $w$ is parallel to $u$. Then either $a,a+w,c,c+w,d,b+w\in \bd K_1$ or $a,a+w,c,c+w,b,d+w\in \bd K_1$. In the first case, by (\ref{cinthya}), it follows that $b\in \inte K_1$ which contradicts that $b\in \bd K_1$. In the other case,  $d\in \inte K_1$ which contradicts that $d\in \bd K_1$.

By the previous paragraph, we know that $w$ is parallel to an edge of $abcd$, say $w$ is parallel to the edge $[a,b]$. If $abcd$ is not a trapezoid, the edge $[c,d]$ is not parallel to $[a,b]$. Let $L$ be a supporting line of $K_1$ parallel to $w$. Then 
\begin{eqnarray}\label{valentina}
[d+L(u)]\cap \inte (abcd)\not= \emptyset \textrm{ } 
\end{eqnarray}
where $L(u):=\{\alpha u: \alpha \in \mathbb{R}\}$. By the relation $K_1=w+K_2$ it is clear that  $a,b,c,d,a+w,b+w,c+w,d+w\in \bd K_1$. On the other hand, by (\ref{valentina}) either $d+w\in \inte K_1$, if $\langle u,  w\rangle>0$ or $d\in \inte K_1$, if $\langle u,  w\rangle<0$. In order to avoid the contradiction we must have $w=0$ and, consequently, $K_1=K_2$.

If $abcd$ is a trapezoid, say the edge $[c,d]$ is  parallel to $[a,b]$, since $a,b,c,d,a+w,b+w,c+w,d+w\in \bd K_1$ and by virtue that $a,a+w,b, b+w\in L(a,b)$ and $c,c+w,d,d+w\in L(c,d)$, the lines $L(a,b)$ and $L(c,d)$ are supporting lines of $K_1$. Analogously we can see that 
the lines $L(a,b)$ and $L(c,d)$ are supporting lines of $K_2$.
Consequently, $[ab]$, $[cd]$ are contained in $\bd M_1$ and $\bd M_2$. 
\end{proof}
\begin{lemma}\label{selene}
The set $A$ is contained in  $\Omega$. 
\end{lemma}
\begin{figure}[h]
\hspace{1cm}
\includegraphics[scale=1.3]{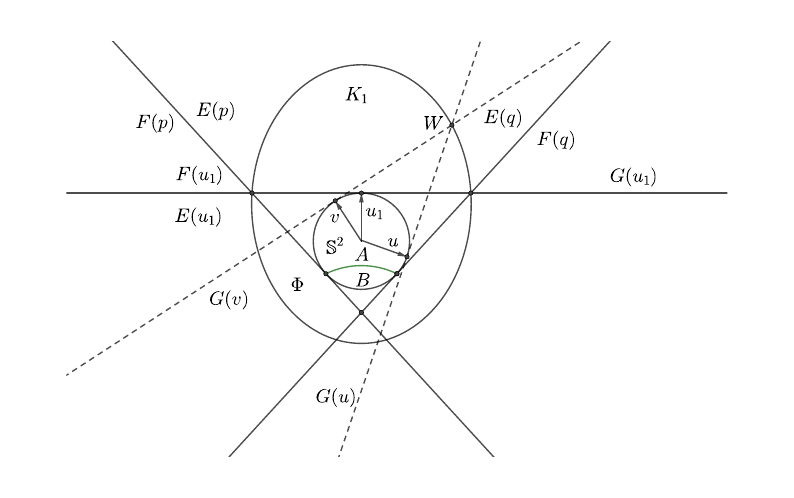}
\caption{The set $A$ is contained in  $\Omega$.}
\label{mariana}
\end{figure}
\begin{proof} 
Let $u\in A$. Then the line $L(u):=G(u)\cap G(u_1)$ intersects the interior of $K_1(u_1)$ (the vectors $u$ which correspond to the lines $L(u)$ which are supporting lines of $K_1(u_1)$ are those in $\Phi$, on the other hand, if $u\in B$, the line $L(u)$ does not intersect $K_1(u_1)$). Let $L(p),L(q)\subset G(u_1)$ be a pair of supporting lines of $K_1(u_1)$, parallel to $L(u)$, such that $L(p)\subset E(u)$ and $L(v)\subset F(u)$ and let $M$ be a supporting line of $G(u)\cap \bd K_1$ parallel to $L(u)$ and contained in $E(u_1)$. We pick a point $x\in M\cap \bd K_1$. 

First we suppose that $G(p) \cap  G(q)\cap \inte K_1=\emptyset$. It follows that $x\in \Sigma$. Let $a,b$ be the extreme points of the line segment $G(u_1)\cap K_1(u)$ and let $c,d$ be the extreme points of the line segment $G(q)\cap K_1(u)$. Since $K_1(u_1)=K_2(u_1)$ and $K_1(q)=K_2(q)$,  notice that $a,b,c,d\in \bd K_1\cap \bd K_2$. If the quadrilateral $\square abcd$ is not a trapezoid, then, by Lemma \ref{parcera},   
$K_1(u)=K_2(u)$. Thus $u\in \Omega$.

If the quadrilateral $\square abcd$ is a trapezoid, by virtue that the segments $[a,b]$, $[c,d]$ cannot belong to the boundary of $K_1(u)$, the segments $[a,c]$, $[b,d]$ are parallel. On the other hand, by virtue that $x\in \Sigma$, it follows that $x\in \bd K_1 \cap \bd K_2$ and there exists a supporting line $N$ of $K_1(u)$ parallel to $L(u)$. Let $w$ be the outer normal vector of the supporting half-plane of $K_1(u)$ defined by $N$. Since $K_1(u)$ and $K_2(u)$ are translated there exists $\alpha \in \Rt$ such that $K_1(u)=\alpha+K_2(u)$. By virtue that $N$ is supporting line of $K_1(u)$ at $x$ and since $x\in \bd K_2(u)$ it follows that $w\cdot \alpha \leq 0$. Since $N$ is supporting line of  $K_1(u)$ at $x$, $N-\alpha$ is supporting line of $K_2(u)$ at $x-\alpha$ but, since the inequality   $w\cdot \alpha \leq 0$ holds, the line $N-\alpha$ separates the point $x$ from the section $K_2(u)$ if $\alpha\not=0$ and $w\cdot \alpha<0$. Thus  if $\alpha\not=0$, then $w\cdot \alpha=0$, but since $[a,c]$ and $[c,d]$ belongs simultaneously to $\bd K_1(u)$ and $\bd K_2(u)$ the condition $w\cdot \alpha=0$ is impossible. Hence $\alpha=0$, i.e., $K_1(u)=K_2(u)$.

Now we suppose that $G(p) \cap  G(q)\cap K_1\not=\emptyset$ (See Fig. \ref{mariana}). If $u$ is such that the relation 
\[
G(u) \cap [E(p)\cap F(q) \cap E(u_1)]\cap \Sigma(K_1,L(p))\not = \emptyset
\]
holds, then we proceed as before and we get to the conclusion $u\in \Omega$. On the other hand, if $u$ is such that the relation
\begin{eqnarray}\label{gitano}
G(u) \cap [E(p)\cap F(q) \cap E(u_1)]\cap \Sigma(K_1,L(p)) = \emptyset
\end{eqnarray}
is satisfied (See Fig. \ref{mariana}), we proceed as
follows. Let $v\in \mathbb{S}^2$ such that the line $W:=G(u)\cap G(v)$ is supporting line of $K_1(u)$, parallel to $L(u)$ and $W\subset  F(u_1)$. We claim that
\begin{eqnarray}\label{luismi}
G(v) \cap [F(p) \cap E(q) \cap \Sigma(K_1,L(p))]\not=\emptyset.
\end{eqnarray} 
We consider the orthogonal projection, defined by $L(u)$, of the sets 
\[
L(p),L(q), G(p) \cap G(q), G(u_1), G(p), G(q),  G(u), G(v),\mathbb{S}^2
\]
which will be denoted, respectively, by 
\[
x,y,z,P,Q,R,S,T, J.
\]
It is clear that $J$ is inscribed in the triangle $\triangle xyz$ and the condition (\ref{gitano}) implies that $S$ intersects the edges $xy$ and $yz$ of the triangle $\triangle xyz$ (See Fig. \ref{mango}). 

Suppose that relation (\ref{luismi}) does not hold, it is equivalent to the fact that $T$ not intersect the edge $xz$ of $\triangle xyz$. Since $T$ intersects the edge $xy$ the aforesaid implies that $T$ intersects $yz$. By virtue that $G(u)$ and $G(v)$ are supporting planes of 
$\mathbb{S}^2$, $J$ is contained in the angle determined by $S$ and $T$.
Thus there is not a point of $J$ in the edge $xz$ but this contradict that $J$ is inscribed in $\triangle xyz$. This contradiction shows that relation (\ref{luismi}) holds.
\begin{figure}[h]
\hspace{1cm}
\includegraphics[scale=.8]{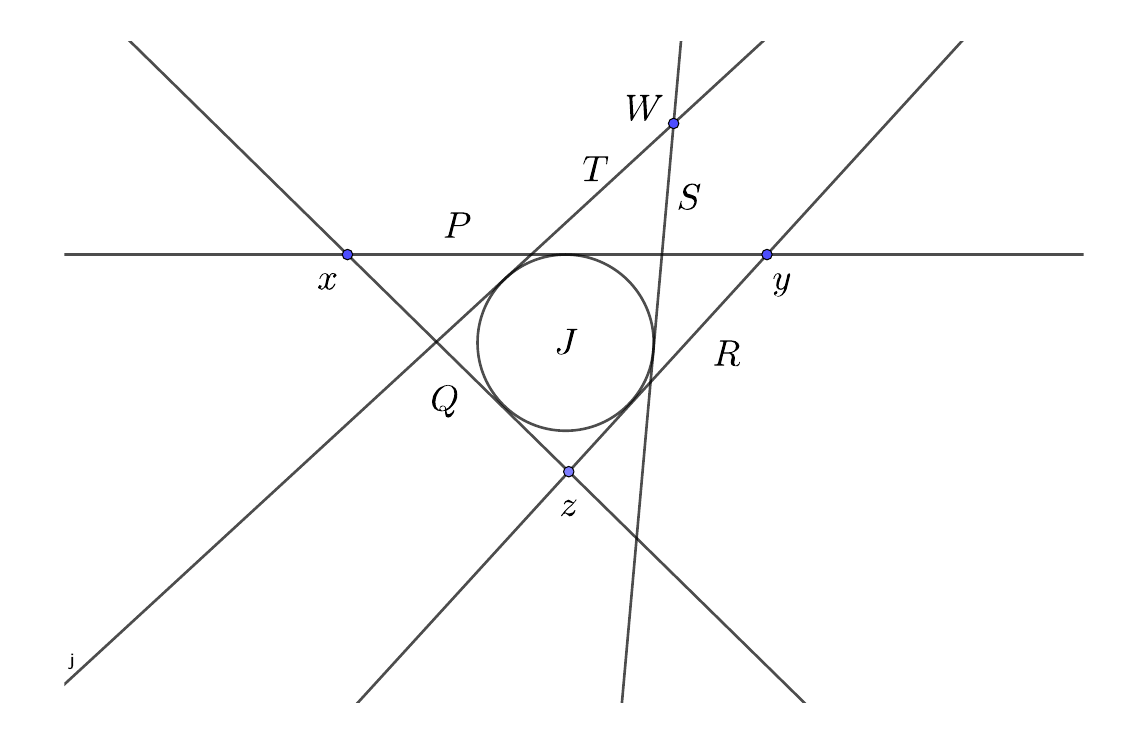}
\caption{The orthogonal projection, defined by $L(u)$, of the sets 
$L(p),L(q), G(p) \cap G(q), G(u_1), G(p), G(q),  G(u), G(v),\mathbb{S}^2$.}
\label{mango}
\end{figure}
From (\ref{luismi}) it follows that, for all $w\in (\mathbb{S}^2\cap \lin\{u_1,u\})$ such that $w\cdot u_1>v\cdot u_1$, the relation
\begin{eqnarray}\label{carrito}
G(w) \cap [F(p) \cap E(q) \cap \Sigma(K_1,L(p))]\not=\emptyset.
\end{eqnarray} 
From (\ref{carrito}),  we conclude that $w\in \Omega$ for all $w\in  (\mathbb{S}^2\cap \lin\{u_1,u\})$ such that $w\cdot u_1>v\cdot u_1$ (See the proof of Lemma \ref{sara}). Consequently, the arc 
$\bd K_1(u)\cap F(u_1)$ belongs to $\bd K_2$. In fact, since $w\in \Omega$ for all $w\in  (\mathbb{S}^2\cap \lin\{u_1,u\})$ such that $w\cdot u_1>v\cdot u_1$, by Lemma \ref{sara}, the points $G(u)\cap  \bd K_1(w)$ belong to $\bd K_2$ as well. Now varying $w\in \mathbb{S}^2$ such that $w\cdot u_1>v\cdot u_1$ we get to the aforesaid. By virtue that the translated sections $K_1(u)$, $K_2(u)$ have the arc $\bd K_1(u)\cap F(u_1)$ in common they coincide. Thus $u\in \Omega$.
\end{proof} 
Let $C$ and $D$ be two convex bodies in $\mathbb{R}^2$, $D \subset \inte C$. For every $x\in \bd C$, we define a polygonal $P_D(x)$, with respect to $C$ and $D$, in the following manner. We take a  supporting line  $L$ of $D$ passing through $x=x_1$, and we denote by $x_2$ the second intersection of $L=L_1$ with $\bd C$.  
We denote now by $L_2$ the supporting line of $C$, $L_2\not= L_1$, passing through $x_2$ and so on. The set of vertices of $P_D(x)$ is the set of points $\{x_1,x_2,x_3,...,x_k,...\}$. Therefore, the edges of $P_{D}(x)$ are  
$\{ \overrightarrow{x_1x_{2}},\overrightarrow{x_2x_{3}},...,\overrightarrow{x_ix_{i+1}},...\}$. Note that $P_D(x)$ is a piece-wise linear curve and could be a closed polygon and, in that case, it could not be convex, furthermore, the sequence $\{x_1,x_2,...\}$ does not necessarily represent a sequence of points in $C$ for which an increasing order of the indices of the sequence corresponds to a direction of travel in $C$.

Given the directed edge $\overrightarrow{x_ix_{i+1}}$, we denote by $E_{i}$ and $F_{i}$ the two half-planes determined by  $L(x_i,x_{i+1})$, we choose the notation such that $E_{i}$ is the half supporting plane of $D$. We denote by $u_i$ the interior unit normal vector of $E_{i}$, $i=1,2,...,k,...$. Now we require that the set $\{\overrightarrow{x_ix_{i+1}}, u_i\}$ is a  left frame of $\mathbb{R}^2$ (this means that if we choose an left orthogonal system of coordinates $(x,y)$ of $\mathbb{R}^2$, then we can find an isometry, which preserves the orientation, such that the images of $\overrightarrow{x_ix_{i+1}}$ and  $u_i$ are parallel to the positive direction of the axis).
\begin{lemma}\label{lizbeth}The relation
\begin{eqnarray}\label{kenia}
\bd C\subset \bigcup_{i\in I} F_{i}
\end{eqnarray}
holds, where $I$ is a countable possibly finite set.
\end{lemma} 
\begin{proof}
Contrary to the statement of Lemma \ref{lizbeth}, we will assume that (\ref{kenia}) does not holds. It is equivalent to suppose the existence of $x\in \bd C$ such that $x_i \rightarrow x$, when $i\rightarrow \infty$, and, furthermore, $x\notin F_{1}$. Notice that $\len [x_i,x_{i+1}]\rightarrow 0$, when $i\rightarrow \infty$. Since in each 
$\overrightarrow{x_ix_{i+1}}$ there is a point $z_i$ in $\bd D$, it follows that $z_i\rightarrow x$, when $i\rightarrow \infty$. Thus $x\in D$. However this contradicts that $D\subset \inte C$.
\end{proof}

\section{Proof of Theorem \ref{paris} in dimension 3 for $r_1=r_2$.}
From Lemma \ref{selene} we conclude $F(u_1) \cap \bd K_1\subset \bd K_2$. Analogously, we can see that, for every $u\in \Omega$, the relation  
\begin{eqnarray}\label{candela}
F(u)\cap \bd K_1\subset \bd K_2
\end{eqnarray} 
holds.

Let $v\in \mathbb{S}^2 \cap H(u_1)$. Let $\mu_1,\mu_2$ be two supporting lines of $K_1(u_1)$ parallel to $v$. We consider the polygonal $P_{\mathbb{S}^2}(x_1)$ inscribed in $\pi_v(K_1)$, it was defined just before stating the Lemma \ref{lizbeth}, where $x_1=\pi_v(\mu_1)$. By Lemma \ref{lizbeth}
\[
\bd \pi_v(K_1)\subset \bigcup_{i\in I(v)} F_{i(i+1)}
\]
where $I(v)$ is a countable set of indices depending of $v$.
Hence
\begin{eqnarray}\label{grecia}
  \bd K_1\subset \bigcup_{i\in I(v)} \pi^{-1}(F_{i(i+1)})
\end{eqnarray}
By virtue that we can find $u_i\in \Omega$ such that $F(u_i)=\pi^{-1}(F_{i(i+1)})$, from (\ref{grecia}) it follows that
\begin{eqnarray}\label{miriam}
\bd K_1\subset \bigcup_{i\in I(v)} F(u_i)
\end{eqnarray}
By (\ref{candela}) and (\ref{miriam}) it follows that $\bd K_1$ and $\bd K_2$ coincides in $\bigcup_{i\in I(v)} F(u_i)$.
Varying $v\in \mathbb{S}^2\cap H(u_1)$ we conclude that $\bd K_1=\bd K_2$.
\section{Lemmas for the case $n=3$ and $r_1\not=r_2$.}  
We take a system of coordinates such that $p_2$ is the origen and $r_2=1$, i.e., $\mathbb{S}^2=S_2$. Since the hypothesis of Theorem \ref{paris} is invariant under translations, if we assume that there exists vectors $u\in \mathbb{S}^2, z\in \Rt$ such that 
\begin{itemize}
\item [(A)] $G(u)$ is supporting plane of $z+S_1$, the sphere $z+S_1$ is contained in $E(u)$ and 
 $K_2(u)=G(u) \cap (z+K_1)$,
\end{itemize}
then, since $S_2\not=z+S_1$, one of the following two conditions holds: 
\begin{itemize}
\item [(B)] $ (z+S_1)  \cap  S_2=u$ and $(z+S_1)\backslash \{u\} \subset \inte S_2$ (See Fig. \ref{celestial}),
\item [(C)] there exists a point $p\in G(u)$ such that $C(z+S_1,p)=C(S_2,p)$ (See Fig. \ref{infernal}).
\end{itemize} 
\begin{figure}[h]
\hspace{1cm}
\includegraphics[scale=.41]{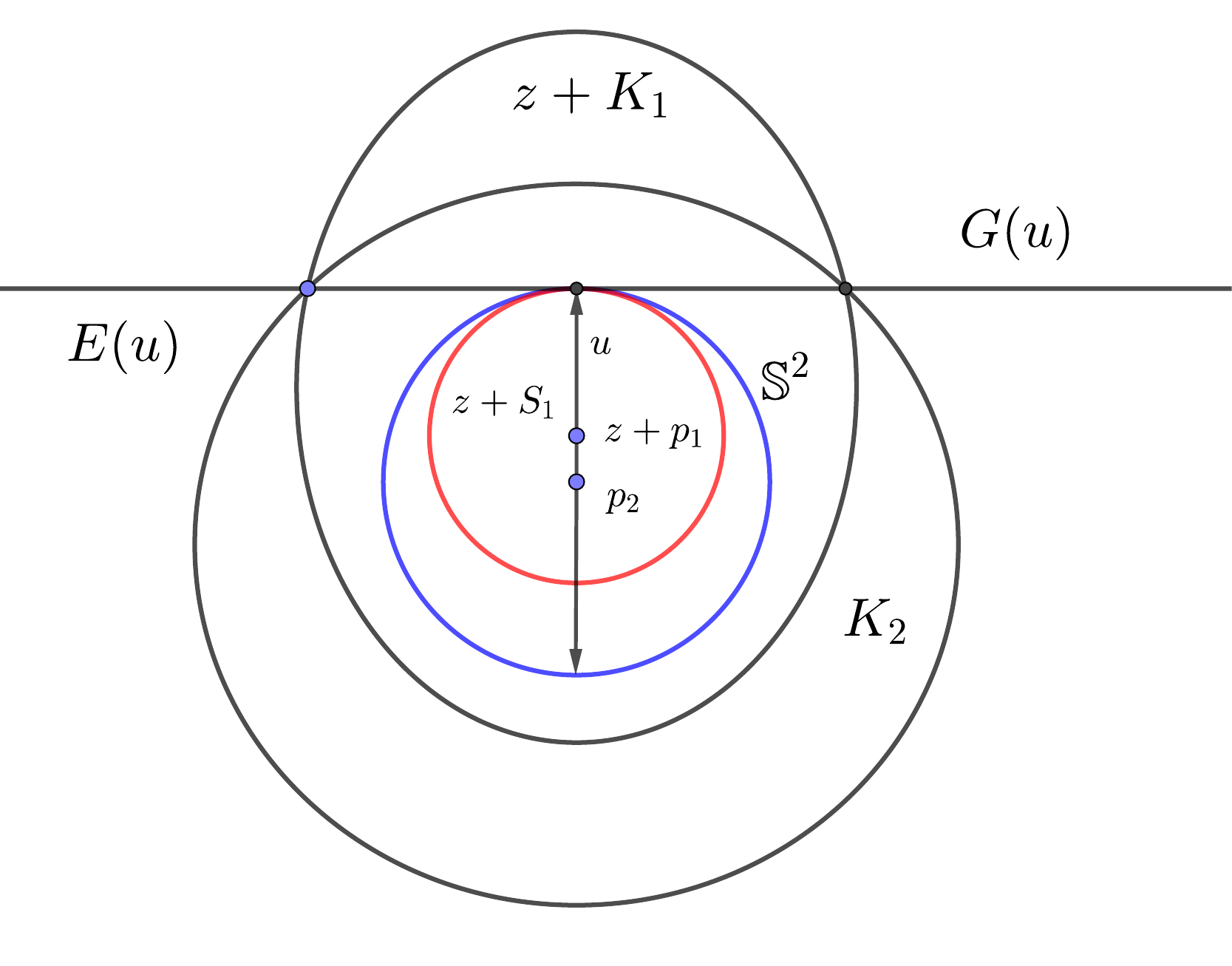}
\caption{Conditions (A) and (B).}
\label{celestial}
\end{figure}
\begin{figure}[h]
\hspace{1cm}
\includegraphics[scale=.41]{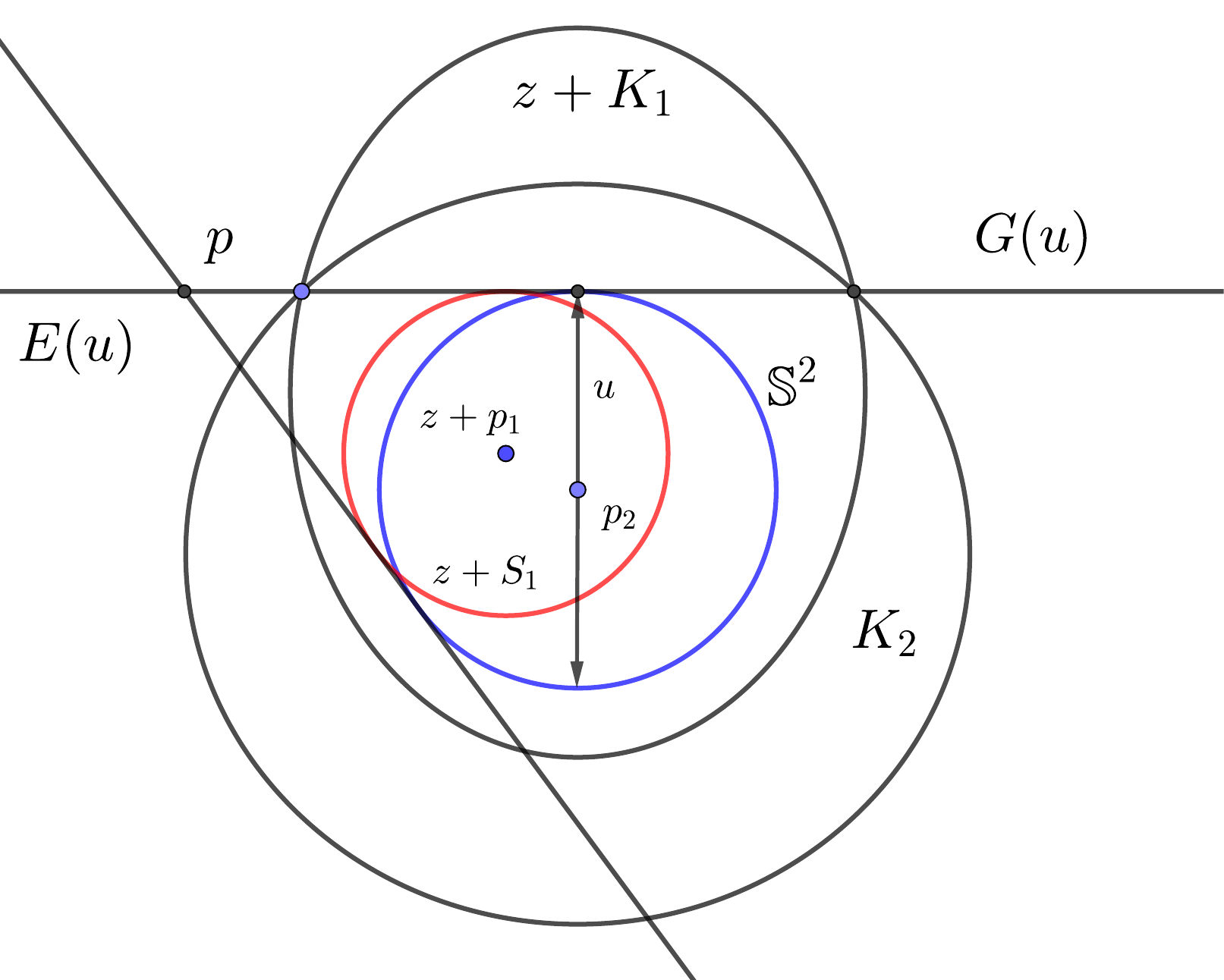}
\caption{Conditions (A) and (C).}
\label{infernal}
\end{figure}
In this section we will suppose that the bodies $K_1,K_2$, the points $p_1,p_2$ and the spheres $S_1,S_2$ are such that there exist vectors $u\in \mathbb{S}^2, z\in \Rt$ for which the conditions (A) and (C) holds. We denote the body $z+K_1$ by $K_1$ and the sphere $z+S_1$ by $S_1$.  
  
Let $x\in \Rt\backslash (S_1\cup S_2)$.  Notice that the collection of corresponding planes of the supporting planes of $C(S_2,x)$ are passing through one point, which will be denoted by $\psi(x)$, and, therefore, they are the envelope of the cone $C(S_1,\psi(x))$. We denote by $\Delta$ the family of supporting planes of $C(S_2,p)$ and by $\Sigma\subset \mathbb{S}^2$ the set of unit vectors such that $u\in  \Sigma$ if $G(u)\in \Delta$. For $u\in \Sigma$, we denote by $K_i (u)$ the section $G(u)\cap K_i, i=1,2$. 

\begin{lemma}\label{chicago}
There is no supporting plane $\Pi $ of $K_1,K_2$ such that $\Pi\notin \Delta$ and \\$\dims[\Pi \cap (K_1\cap K_2)]=2$. 
\end{lemma}
\begin{figure}[h]
\hspace{1cm}
\includegraphics[scale=.45]{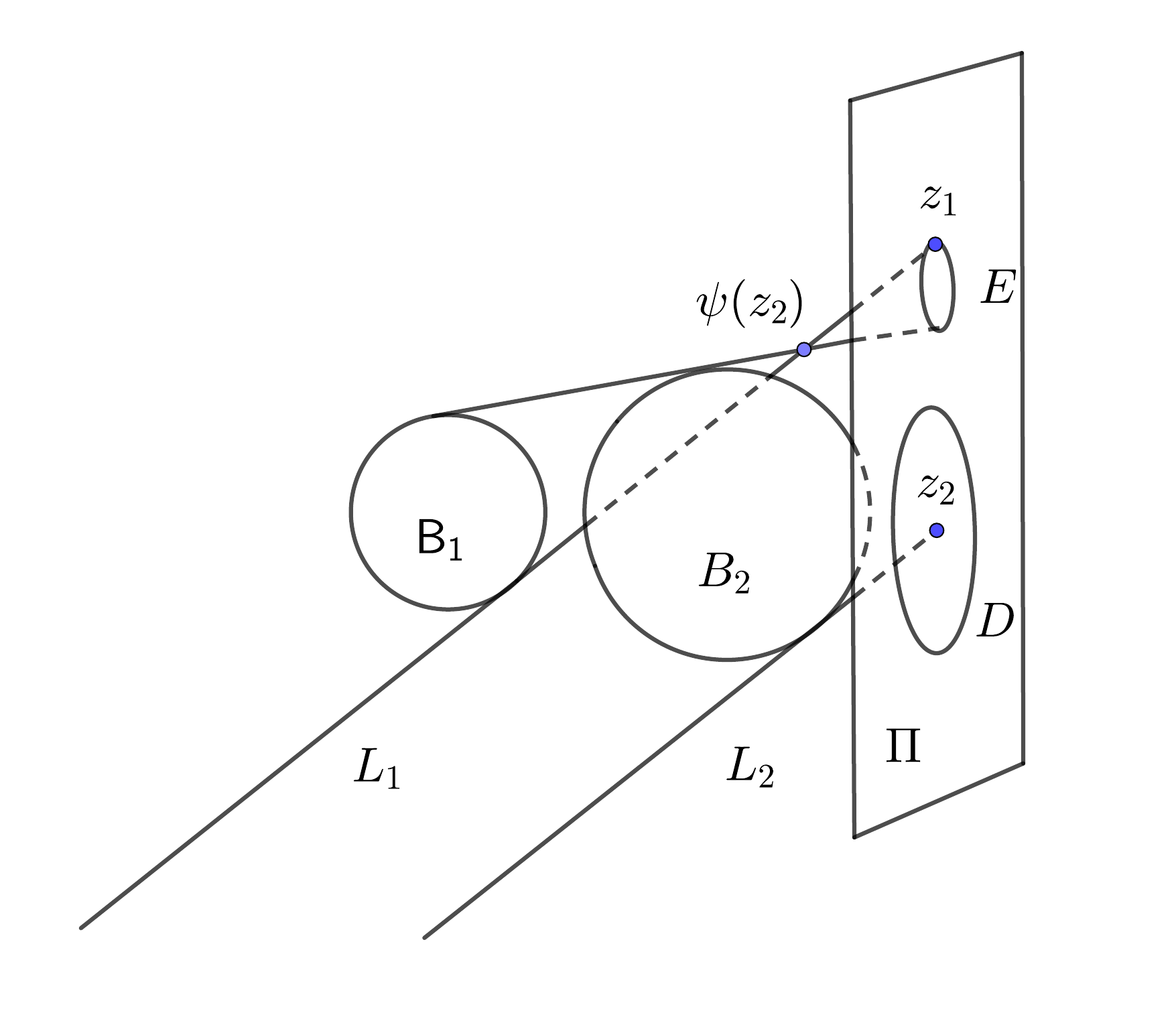}
\caption{There is no supporting plane $\Pi $ of $K_1,K_2$ such that $\Pi\notin \Delta$ and $\dims[\Pi \cap (K_1\cap K_2)]=2$.}
\label{casi}
\end{figure}
\begin{proof}
Contrary to the statement of Lemma \ref{chicago}, suppose that there exists a supporting plane $\Pi$ of $K_1,K_2$ such that $\dims[\Pi \cap (K_1\cap K_2)]=2$, $\Pi \notin \Delta$. We denote by $A,B$ the sets $\Pi \cap K_1$, $\Pi \cap K_2$, respectively and let $D\subset \Pi$ be a disc such that $A\cup B\subset D$. Let $L_1,L_2$ be a pair of supporting lines of $S_1, S_2$, respectively, such that $L_1,L_2$ are parallel and $z_1:=L_1\cap \Pi \notin D$ and $z_2:=(L_2\cap \Pi)\in A\cap B$ (See Fig. \ref{casi}). Let $\Gamma_1,\Gamma_2$ be a pair of corresponding planes of $S_1,S_2$ passing through $\psi(z_2)$ and $z_2$ such that $\Gamma_1 \cap D=\emptyset$. By the hypothesis, there exists $\alpha\in \Rt$ such that $\Gamma_1\cap K_1= \alpha+(\Gamma_2\cap K_2)$. We claim that $\alpha$ cannot be parallel to $\Pi$. Otherwise, the equality $(\Gamma_1\cap K_1)\cap \Pi= \alpha+[(\Gamma_2\cap K_2)\cap \Pi]$ would hold. Consequently, $\alpha+z_2\in (\Gamma_1\cap K_1)\cap \Pi=\Gamma_1\cap (K_1\cap \Pi)=\Gamma_1\cap A$. By virtue that $A\subset D$ and $\Gamma_1\cap D=\emptyset$ it follows that $\Gamma_1 \cap A=\emptyset$. This contradiction proves our claim.

As a corollary we obtain that line segment $\alpha+[(\Gamma_2\cap K_2)\cap \Pi]$ is contained in $\bd K_1\backslash \Pi$ and it is parallel to $\Pi$.

By the choice of $z_1$ and $z_2$ it is possible to take $\Gamma_1, \Gamma_2$ in such a way that we can construct a family of line segments (non-degenerated in a point) contained in $\bd K_1\backslash \Pi$, parallel to $\Pi$ and whose corresponding set of direction is a planar set in $\mathbb{S}^2$ and whose measure is not zero, however, this contradicts the Theorem 1 of \cite{ewald}. Such contradiction shows that our initial assumption is false. From this the Lemma follows. 
\end{proof}

Now we suppose that $p\notin K_2(u)$. Let $L_1 \subset G(u)$ be a supporting line of $ K_2(u)$, $p\in L_1$. In virtue that the points $q_1=G(u)  \cap S_1$, $q_2= G(u) \cap S_2 $ belongs to $\inte K_2(u)$, $L_1$ is not contained in $C(S_2,p)$. Thus there exists a supporting plane $G(v)$ of $C(S_2,p)$ such that $L_1\subset G(v)$ and $G(u)\not=G(v)$.  
\begin{lemma} \label{mambo} 
The equality 
\begin{eqnarray}\label{alex}
G(v) \cap K_1=G(v) \cap K_2
\end{eqnarray} 
holds.
\end{lemma}
 \begin{proof} 
Using the same argument that in the proof of Lemma \ref{caboe} the equality (\ref{alex}) follows.    
\end{proof}
 For the sets $K_i^p:=(\mathbb{R}^n \backslash \inte C(K_i,p)) \cap \bd K_i$, respectively, $i=1,2$, where the point $p \in \mathbb{R}^n$ is given by the condition (C), we have the following lemma.
\begin{lemma} \label{choco} 
The equality $K_1^p=K_2^p$ holds.
\end{lemma}
 \begin{proof}  
 Let $q$ be a point in $L_1\cap   K_2(u)$. We take a line $L_{\tau}\subset G(u)$ such that $p\in L_{\tau}$ and it has an interior point $\tau$ of $K_2(u)$ in the line segment $[q_1q]$, $q_1\not=\tau\not=q$. Notice that, since $q_1 \in \inte K_2$ and since all the points of $ K_2(u)\backslash \bd K_2$ are interior points of $K_2$, $\tau\in \inte K_2$. The line $L_{\tau}$ is not contained in $S(K_2,p)$ thus there exists a supporting plane $G(w)$ of $S(K_2,p)$ so that $L_{\tau}\subset G(w)$, $G(w) \not= G(u)$ and $G(w) \not= G(v)$. Then, by Lemma \ref{mambo}, the sections $ K_1(w)$ and $K_2(w)$ have four points in common, the points given by the intersection of the lines $L_{\tau}$ and $G(w) \cap G(v)$ with $\bd K_2$. We denote by $a,b$ and $c,d$ such points, respectively. By hypothesis   there exists a vector $\alpha \in G(w)$ such that $K_1(w)= \alpha+ K_2(w)$. Then either: $\alpha=0$, i.e., $ K_1(w)=  K_2(w)$ or $\alpha \not=0$. If $\alpha=0$, we finish. 
\begin{figure}[h]
\hspace{1cm}
\includegraphics[scale=.4]{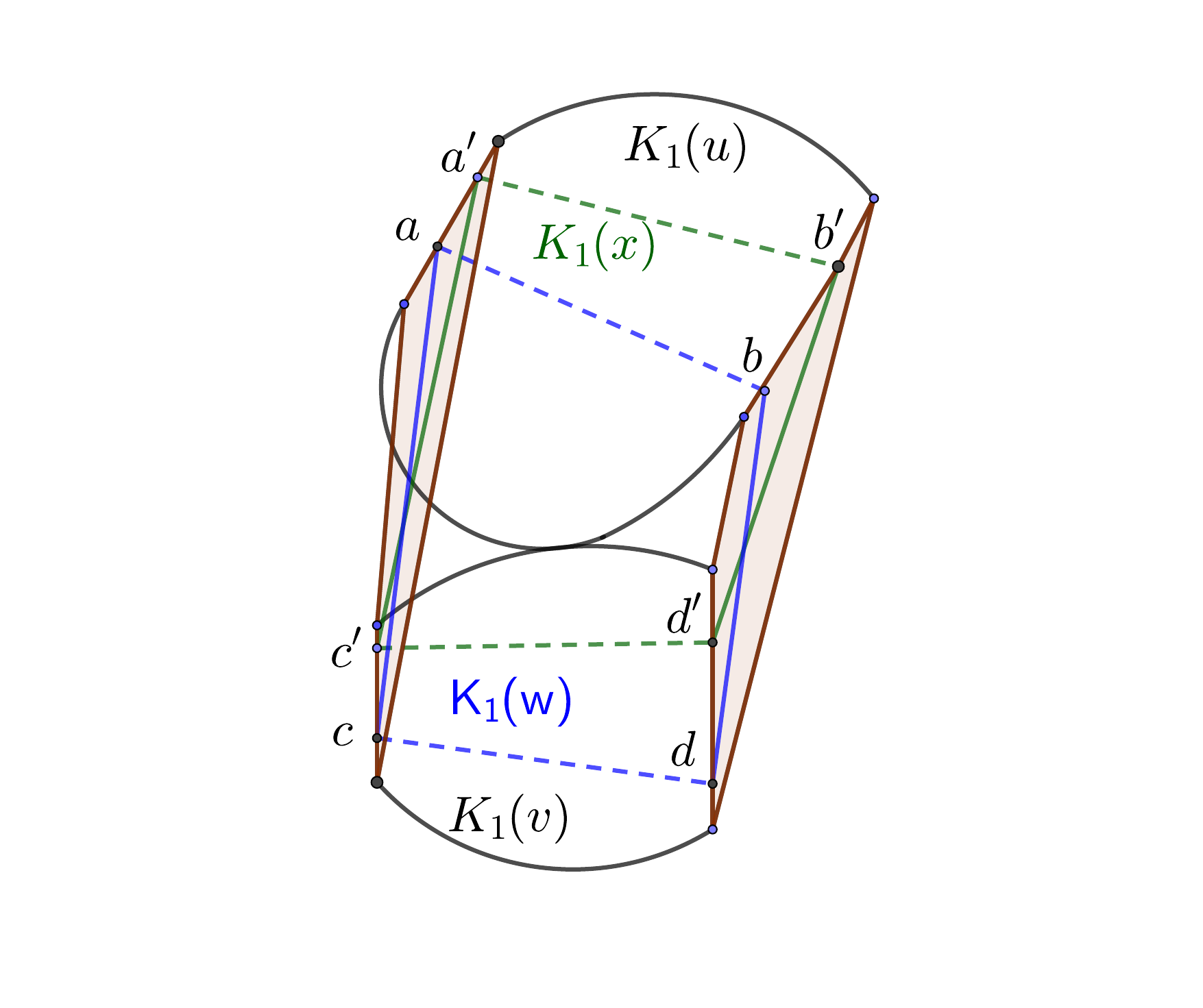}
\caption{The equality $K_1^p=K_2^p$ holds.}
\label{ya}
\end{figure}
 On the other hand, we suppose now that  $\alpha \not=0$. Then, by Lemma \ref{parcera}, the quadrilateral $ \square abcd$ is a trapezoid and it has one pair of edges parallel to $\alpha$. If the line segment $ab$ is parallel to $\alpha$, then $ab \subset \bd K_2$, however it would contradict that $\tau \in ab$ is an interior point of $K_2$. Thus the edges $bc$ and $da$ are parallel to $\alpha$ and  
$bc, da\subset \bd K_1\cap \bd K_2$. We take a line $L_{\rho}\subset G(u)$ such that $p\in L_{\rho}$ and it has an interior point $\rho$ of 
$ K_2(u)$ in the line segment $q_1q$, $\tau\not= \rho$. The line $L_{\rho}$ is not contained in $S(K_2,p)$ thus there exists a supporting plane $G(x)$ of $S(K_2,p)$ such that $L_{\rho}\subset G(x)$, $G(x) \not= G(u)$, $G(x) \not= G(v)$ and $G(x) \not= G(w)$. The sections $ K_1(x)$ and $ K_2(x)$ are translated and have six points in common, the points given by the intersection of the lines $L_{\rho}$, $G(x) \cap G(v)$ and $G(x) \cap G(w)$ with $\bd K_1$. We denote by $a',b'$, $c', d'$ and $e', f'$ such points, respectively. As we have seen above the segments $a',b'$, $c', d'$ and $e', f'$ cannot be parallel. If the quadrilateral $\square a'b'c'd'$ is a trapezoid with $a'c'$ and $b'd'$ a pair of edges parallel, then the points $\{a,a',c,c'\}$ and $\{b,b',d,d'\}$ defined a pair of parallel supporting planes $\Pi_1,\Pi_2$ of $K_2$ and $\conv \{a,a',c,c'\},\conv \{b,b',d,d'\}\subset \bd K_1\cap \bd K_2$ (See Fig. \ref{ya}).  However, according Lemma \ref{chicago}, this is impossible. Thus $\square a'b'c'd'$ is not a trapezoid and, consequently, $ K_1(x)= K_2(x)$.

As the supporting planes $G(s)$ of $C(S_2,p)$, $s\in \textrm{arc} (u,v)\subset \Sigma$ are in correspondence with the points $t\in [p_1,q], t\not=\tau$ we get that for every $s\in \textrm{arc} (u,v)\subset \Sigma$ the equality $G(s)\cap K_1=G(s)\cap K_2$ holds. 

Finally we can extent the previous argument for all $s\in \Sigma$ 
(considering the second supporting line of $ K_2(v)$ passing through $p$ we find a vector $z\in \Sigma$ such that $  K_1(z)= K_2(z)$ and repeat the argument from above to $ K_2(v)$ and $ K_2(z)$ and so on until we cover the set $\Sigma$). Now the proof of Lemma \ref{choco} is complete.  
\end{proof}
\begin{lemma}\label{ade}
It is not possible to have $K_1=K_2$ and $S_1$ not contained in $S_2$ in Theorem \ref{paris}, i.e., there are no convex body $K\subset \mathbb{R}^{3}$, two spheres $S_1,S_2\subset \inte K$ of radii $r_1, r_2$, $r_1<r_2$, $S_1$ is not contained in $S_2$, such that, for every pair of corresponding planes $\Pi_1, \Pi_2$ of $S_1$ and $S_2$, there exists a translation $\psi: \mathbb{R}^{n} \rightarrow \mathbb{R}^{n}$ such that
\[
\psi(\Pi_2\cap K) = \Pi_1\cap K.
\]	 
\end{lemma}
\begin{proof}
Contrary to the statement of the Lemma \ref{ade}, let us assume that there are convex body $K\subset \mathbb{R}^{3}$, two spheres $S_1,S_2\subset \inte K$ of radius $r_1, r_2$, $r_1<r_2$, $S_1$ is not contained in $S_2$, such that, for every pair of corresponding planes $\Pi_1, \Pi_2$ of $S_1$ and $S_2$, there exists a translation $\psi: \mathbb{R}^{n} \rightarrow \mathbb{R}^{n}$ such that $
\psi(\Pi_2\cap K) = \Pi_1\cap K$.	 

Since the measure of the set of directions parallel to the segments contained in $\bd K$ is zero, according with \cite{ewald}, we can fin a direction $w\in \Sd$ and a pair of corresponding planes $\Pi_1,\Pi_2$ such that there is no a line segment parallel to $w$, contained in $\bd K$, and $\Pi_1\cap K=w+(\Pi_2 \cap K)$. Notice that
\begin{eqnarray}\label{podromu}
\textrm{area}(\Pi_1\cap K)=\textrm{area}(\Pi_2\cap K),
\end{eqnarray}
where $\textrm{area}(Z)$ denotes the area of a set $Z\subset \Rd$.

Let $\Gamma_1, \Gamma_2$  be a pair of corresponding planes with $\Gamma_1$ parallel to $\Pi_1$. Thus there exists $\bar{w}\in \Sd$ such that $\Gamma_1\cap K=\bar{w}+(\Gamma_2\cap K)$ which implies that 
\begin{eqnarray}\label{citlali}
\textrm{area}(\Gamma_1\cap K)=\textrm{area}(\Gamma_2\cap K).
\end{eqnarray}
On the other hand, since the planes $\Pi_1,\Pi_2$  are corresponding there exists $u\in \mathbb S^{3}$ such that 
\[
\Pi_1=p_1+r_1G(u)\textrm{      }\textrm{      }\textrm{     and    }\textrm{      }\textrm{      }S_1\subset p_1+r_1E(u) 
\]
\[
\Pi_2=G(u)\textrm{      }\textrm{      }\textrm{     and    }\textrm{      }\textrm{      }S_2\subset E(u)  
\]
(notice that $p_2$ is the origin and $r_2=1$).

First, we assume that $S_1\cap S_2\not=\emptyset$. Then 
\begin{eqnarray}\label{anton}
\Pi_1\cap K\subset  F(u) 
\end{eqnarray}
and
\begin{eqnarray}\label{plenitud}
\Gamma_1\cap K, \Gamma_2 \cap K\subset E(u) \textrm{  }\textrm{(See Fig. \ref{end1}).}
\end{eqnarray}
By (\ref{podromu}), (\ref{anton}), (\ref{plenitud}), the choice of $w$ and the Brunn's inequality for slice volumes (Theorem 12.2.1 P. 297 of \cite{matousek})
\begin{eqnarray}\label{magali}
\textrm{area}(\Gamma_1\cap K)\not=\textrm{area}(\Gamma_2\cap K).
\end{eqnarray}
However (\ref{magali}) contradicts (\ref{citlali}). 

The case $S_1\cap S_2=\emptyset$ can be considered analogously.
Thus the proof of Lemma \ref{ade} is now complete.
\end{proof}
\begin{figure}[h]
\hspace{1cm}
\includegraphics[scale=1.3]{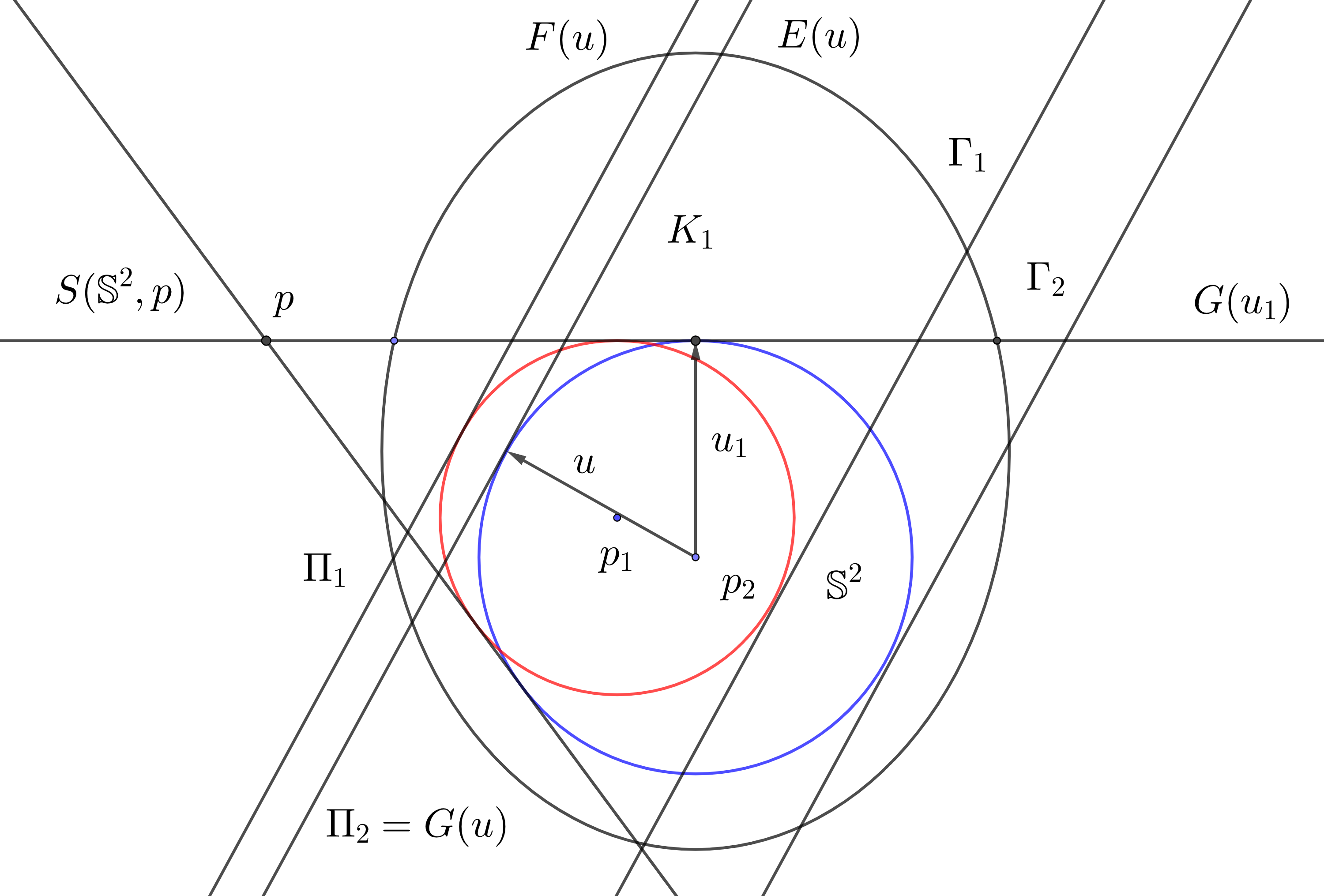}
\caption{It is not possible for $K_1=K_2$ and $S_1$ not contained in $S_2$ in Theorem \ref{paris}.}
\label{end1}
\end{figure}
\section{Proof of Theorem \ref{paris} in dimension 3 for $r_1\not=r_2$.}
\subsection{Conditions (A) and (B) are satisfied simultaneously.}  We suppose that the bodies $K_1,K_2$ are such that every time we translate the body $K_1$ and in order that the condition (A) holds, then the the condition (B) holds. This supposition implies that, since the conditions of the Theorem \ref{paris} are invariant under translations, if the spheres $S_1$  has its center at the origin, then, for all $u\in \Sd$, it follows that
\[
K_2(u) =(1-r_1)u+[r_1G(u) \cap K_1] , \textrm{ }\textrm{ }((1-r_1)u+S_1) \cap   S_2=u 
\textrm{ and }  [(1-r_1)u+S_1]  \subset S_2, 
\]
i.e, it is possible to translate $K_1$ and $S_1$ in order that (A) and (B) holds for all $u \in \Sd$ (See Fig. \ref{ivan}). Let $\Delta: \Rt \rightarrow \Rt$ be a homothety with centre at the origin and coefficient of homothety $r=r_1$. Thus the bodies $K_1$ and $\Delta(K_2)$ have the following property: for each supporting plane $\Pi$ of $S_1$ the sections $\Pi \cap K_1$ and $\Pi \cap \Delta(K_2)$ are homothetic with centre and radius of homothety $\Pi \cap S_1$ and $r$, respectively, i.e., the bodies $K_1$ and $\Delta(K_2)$ satisfies the condition of Lemma \ref{chaparrito} and, consequently, $K_1$ and $\Delta(K_2)$ are two concentric spheres. Therefore $K_1$ and $K_2$ are two concentric spheres.
\begin{figure}[h]
\hspace{1cm}
\includegraphics[scale=.6]{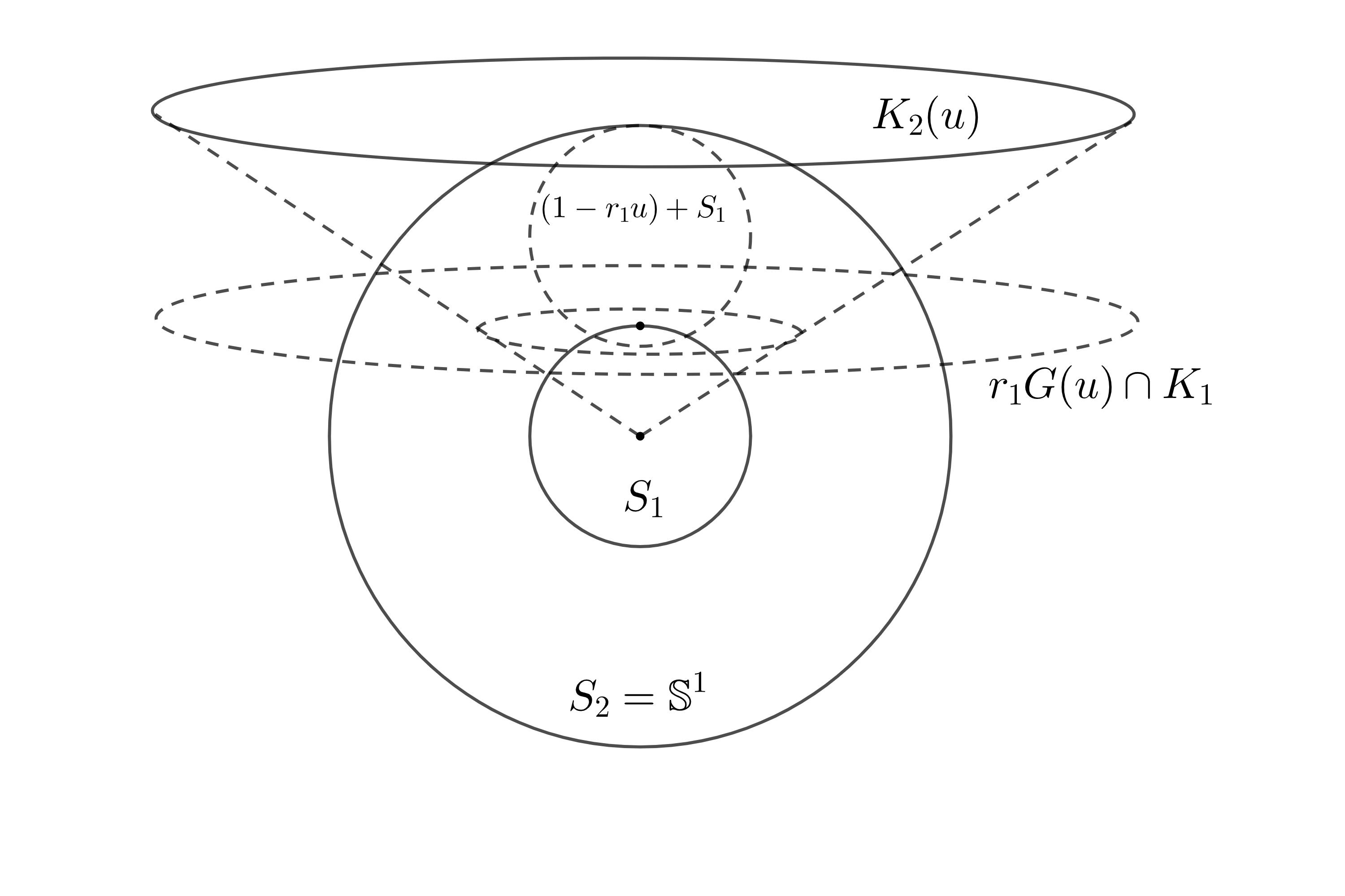}
\caption{Conditions (A) and (B) are satisfied simultaneously.}
\label{ivan}
\end{figure}\subsection{The bodies $K_1,K_2$, the points $p_1,p_2$ and the spheres $S_1,S_2$ are such that there exist vectors $u\in \mathbb{S}^2, z\in \Rt$ for which the conditions (A) and (C) hold.}

By Lemma \ref{choco}, the equality $K_1^p=K_2^p$ holds, however, by Lemma \ref{ade}, $\bd K_1\not=\bd K_2$. Thus there exists a pair of supporting planes of $\Gamma_1$, $\Gamma_2$ of $K_1$ and $K_2$ which are corresponding, close enough to $G(u)$, such that there exists $\bar{z}\in \Rt$ with the properties
\begin{itemize}
 \item [($\mathcal{A})$] $\Gamma_2 \cap K_2=\bar{z} +(\Gamma_1 \cap K_1)$, $\Gamma_2$ is supporting plane of $\bar{z}+S_1$ and the sphere $\bar{z}+S_1$ is contained in $\Gamma^+_2$, the supporting half-space of $S_2$ defined by $\Gamma_2$,  
  ,
 \item [($\mathcal{C})$] there exists a point $\bar{p}\in \Gamma_2\backslash((\bar{z}+K_1)\cup K_2)$ such that $C(\bar{z}+S_1,\bar{p})=C(S_2,\bar{p})$ (See Fig. \ref{infernal}).
\end{itemize}
Furthermore, by Theorem 1 of \cite{ewald}, we can take the planes $\Gamma_1$, $\Gamma_2$ in such a way that
\begin{itemize}
\item [($\mathcal{D}$)]
there is no line segment parallel to $\bar{z}$ and contained neither in $\bd K_1$ nor  $\bd K_2$.
\end{itemize}

We will consider the cylinders 
\[
C:=\{\lambda \bar{z}+x:x\in \Gamma_1\cap K_1, 0\leq \lambda \leq 1\},  C^-:=\{\lambda \bar{z}+x:x\in \Gamma_1\cap K_1,   \lambda <0\},
\] 
the planes $\Gamma_{\lambda} :=\lambda \bar{z}+\Gamma_1$, $\Delta:=\{ x+\lambda \bar{z}: x\in G(u)\cap \Gamma_2, \lambda \in \R\}$ and the half-spaces $\Delta^+, \Delta^-$ defined by $\Delta$, we choose the notation such that $p\in \Delta^-$. Furthermore, we establish the notation $\bar{K}_1=K_1+\bar{z}$,  $\bar{K}_2:=K_2$, 
$\bar{K}^{\bar{p}}_i:=[\Rt\backslash C(\bar{K}_i, \bar{p})]\cap \bd \bar{K}_i$, $i=1,2$.

By Lemma \ref{choco}, the following relations hold
\begin{eqnarray}\label{tierna}
\bd (\Gamma_1 \cap K_1)\cap F(u)\subset \bd K_1, \bd K_2 \textrm{   } \textrm{ and } \textrm{    }\bd (\Gamma_2 \cap K_1)\cap F(u)\subset \bd K_1, \bd K_2.
\end{eqnarray}
By (\ref{choco}), ($\mathcal{D}$) and the Brunn's inequality for slice volumes (Theorem 12.2.1 P. 297 of \cite{matousek}), the section $\Gamma_{\lambda} \cap K_1$ is such that 
\begin{eqnarray}\label{mojadita}
A_{\lambda}:=[\bd (\Gamma_{\lambda} \cap K_1)] \cap \Delta^+ \subset \Rt \backslash C,
\end{eqnarray}
for every $\lambda$, $0< \lambda < 1$, and
\begin{eqnarray}\label{placer}
C_{\lambda}:=[\bd(\Gamma_{\lambda} \cap K_1)\cap \Delta^+] \subset C^-\cap \Delta^+, 
\end{eqnarray}
for every $\lambda$, $\lambda<0$ such that 
$\Gamma_{\lambda} \cap (K_1 \cap \Delta^+)\not=\emptyset$ (See Fig. \ref{patricio}).
Notice that by Lemma \ref{choco} 
\begin{eqnarray}\label{sabrosita}
A_{\lambda}=B_{\lambda}:= [ \bd (\Gamma_{\lambda} \cap K_2)] \cap \Delta^+\textrm{   and    }C_{\lambda}=D_{\lambda}:= [ \bd (\Gamma_{\lambda} \cap K_2)] \cap \Delta^+.
\end{eqnarray}
for every $\lambda$, $0< \lambda < 1$. From (\ref{mojadita}) and (\ref{sabrosita})
\begin{eqnarray}\label{seco}
B_{\lambda}  \subset \Rt \backslash C,
\end{eqnarray}
for every $\lambda$, $0< \lambda < 1$, and, on the other hand, from (\ref{placer}) and (\ref{sabrosita})
\begin{eqnarray}\label{dolor}
D_{\lambda}  \subset C^-\cap \Delta^+, 
\end{eqnarray}
for every $\lambda$, $\lambda<0$.
\begin{figure}[h]
\hspace{1cm}
\includegraphics[scale=.45]{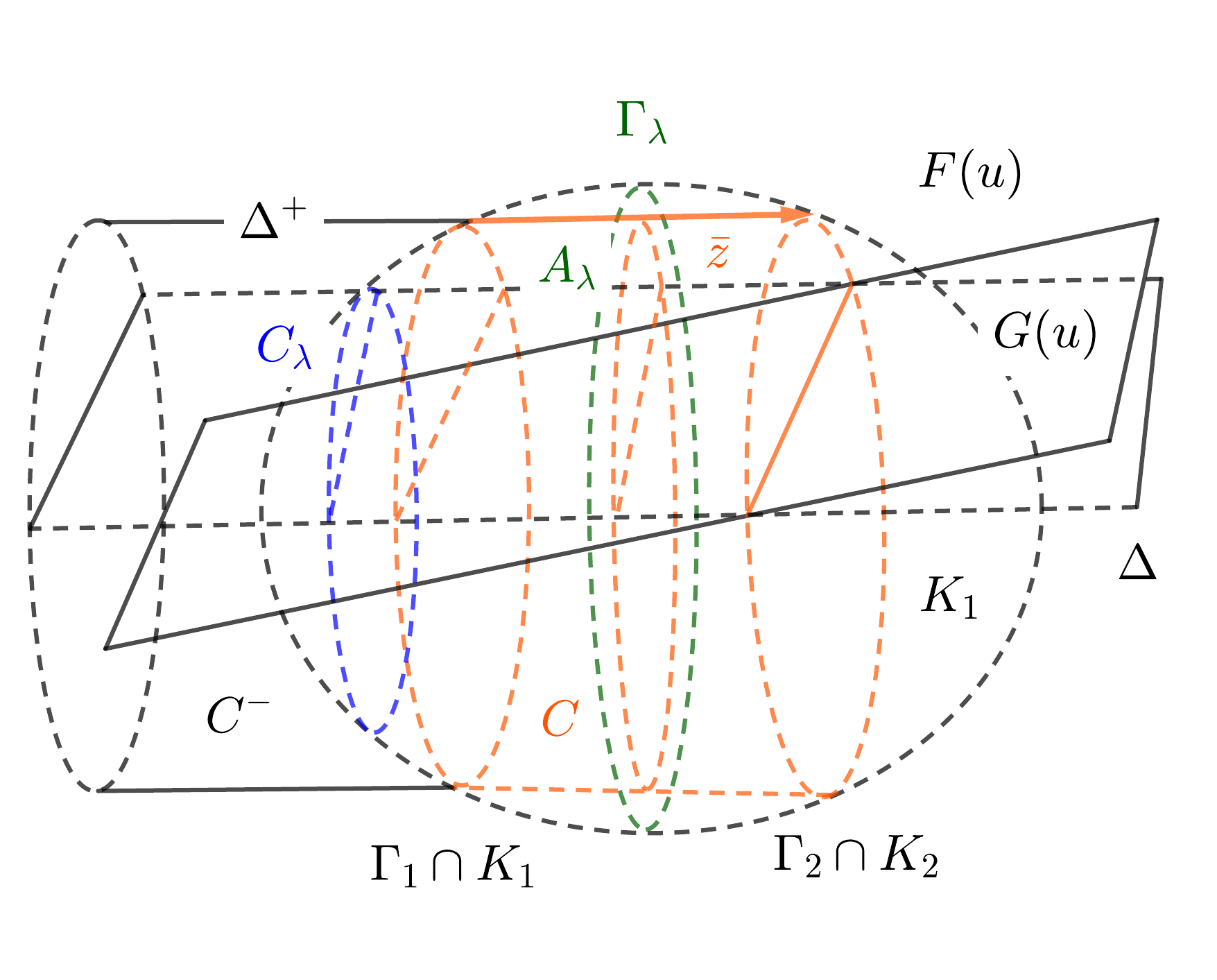}
\caption{The conditions (A) and (C) implies a contradiction}
\label{patricio}
\end{figure}
By (\ref{placer}), for $\epsilon>0$ small enough such that $\Gamma_{-\epsilon} \cap (K_1 \cap \Delta^+)\not=\emptyset$, the relation $C_{-\epsilon} \subset C^-\cap \Delta^+$ holds, consequently 
\begin{eqnarray}\label{blanca}
C_{-\epsilon}+\bar{z}\subset C\cap \Delta^+.
\end{eqnarray}
Notice that $C_{-\epsilon}+\bar{z}\subset \bd \bar{K}_1$. 
By Lemma \ref{choco} applied to the bodies $\bar{K}_1,\bar{K}_2$, by virtue that we are assuming that the conditions ($\mathcal{A})$ and ($\mathcal{C})$ are satisfied, the equality $\bar{K}^{\bar{p}}_1=\bar{K}^{\bar{p}}_2$ holds. Thus $C_{-\epsilon}+\bar{z}\subset \bd \bar{K}_2=K_2$. Since $\Gamma_{-\epsilon}+\bar{z}=\Gamma_{(1-\epsilon)}$ it follows that $C_{-\epsilon}+\bar{z}=B_{(1-\epsilon)}$. Hence \ref{blanca} implies 
\begin{eqnarray}\label{negra}
B_{(1-\epsilon)}\subset C\cap \Delta^+.
\end{eqnarray}
By (\ref{seco}) 
\begin{eqnarray}\label{bebe}
B_{(1-\epsilon)}\subset (\Rt \backslash C).
\end{eqnarray}
From (\ref{negra}) and (\ref{bebe}) we have an absurd. This contradiction was derived by the assumption that the conditions (A) and (C) hold.


\begin{thebibliography}{9}
\bibitem{bl} J. A. Barker, D.G. Larman: \emph{Determination of convex bodies by certain sets of sectional volumes}. Discrete Mathematics {\bf 241} (2001) $79-96$.

\bibitem{bg} G. Bianchi and P. Gruber: \emph{Characterization of Ellipsoids},
Archiv der Math. {\bf 49} (1987) $344-350$.

\bibitem{burton} G. R. Burton. \emph{Sections of convex bodies}. J. London Math. Soc., Vol {\bf 2}-12, (1976) $331-336$.


\bibitem{bu} H. Busemann, \emph{The geometry of geodesic}, New York, 1955.


\bibitem{ewald} G. Ewald, D. G. Larman, C. A. Rogers, \emph{The directions of the line segments and of the $r$-dimensional balls on the boundary of a convex body in Euclidean space}. Mathematika Vol. {\bf17}-1,(1970), $1-20$.

\bibitem{gardner} R. Gardner. \emph{Geometric Tomography} (Second Edition). Cambridge University Press. USA. 2006. 


\bibitem{jmm2} J. Jeronimo-Castro, L. Montejano and E. Morales. \emph{Shaken Roger's theorem for homothetic sections}. Canadian mathematical bulletin, ISSN 0008-4395, Vol. {\bf 52}, No. 3 (2009) $403-406$.
\bibitem{kuru1} \'{A}. Kurusa, T. \'{O}dor. \emph{Characteriztion of balls by sections and caps}. Beitr. Algebra Geom. {\bf 56} (2015), no. 2, 459-471.
\bibitem{kuru2} \'{A}. Kurusa, T. \'{O}dor. \emph{Isoptic characterization of the sphere.} J. Geom. {\bf 106} (2015), no. 1, 63-73.

\bibitem{matousek} J. Matousek. \emph{Lectures on discrete geometry} (2002). Springer


\bibitem{mo} L. Montejano: \emph{Two applications of Topology to Convex Geometry}.
Proceedings of the Steklov Institute of Mathematics, Vol. {\bf 247} (2004) $164-167$.

\bibitem{dima} F. Nazarov, D. Ryabogin, A. Zvavitch.  \emph{Non-uniqueness of convex bodies with prescribed volumes of sections and projections}. Mathematika, Vol. {\bf 59}, Part 1 (2013) $213-221$.

\bibitem{ro1} C. A. Rogers. \emph{Sections and projections of convex bodies}, Portugaliae Math. {\bf 24} (1965) 99-103.


\bibitem{steenrod} N. Steenrod: \emph{Topology of fibre bundles}. Princeton Landmarks in Mathematics. 1960.

\end{thebibliography}
\end{document}